\documentclass{kms-b}

\issueinfo{}
  {}
  {}
  {}
\pagespan{1}{}
\copyrightinfo{}
  {Korean Mathematical Society}

\usepackage{graphicx}
\allowdisplaybreaks
\usepackage{amsmath}
\usepackage{mathrsfs}
\usepackage{amsfonts}
\usepackage{amssymb}
\usepackage{amsthm}
\usepackage{graphicx}
\usepackage{grffile}
\usepackage{caption}
\usepackage{subcaption}
\usepackage[all]{xy}
\SelectTips{cm}{12}
\UseTips{}
\usepackage{euscript}
\usepackage{latexsym}
\usepackage{epsf}
\theoremstyle{plain}
\newtheorem{theorem}{Theorem}[section]
    \newtheorem*{thma}{Theorem A}
    \newtheorem*{thmb}{Theorem B}
    \newtheorem*{thmc}{Theorem C}
    \newtheorem*{propd}{Proposition D}
\newtheorem{proposition}[theorem]{Proposition}
\newtheorem{lemma}[theorem]{Lemma}
\newtheorem{corollary}[theorem]{Corollary}

\theoremstyle{definition}
\newtheorem*{definition}{Definition}
\newtheorem{example}[theorem]{Example}
    \newtheorem{notation}[theorem]{Notation}

\theoremstyle{remark}
\newtheorem{remark}[theorem]{Remark}
\newcommand{\w}[2][ ]{\ \ensuremath{#2}{#1}\ }
\newcommand{\ww}[2][ ]{\ \ensuremath{#2}{#1}}
\newcommand{\wwd}[1]{\ \ensuremath{#1}-}
\newcommand{\wh}{\ -- \ }
\newcommand{\wwh}{-- \ }
\newcommand{\wb}[2][ ]{\ (\ensuremath{#2}){#1}\ }
\newcommand{\wwb}[1]{\ (\ensuremath{#1})-}
\newcommand{\www}[2][ ]{\ensuremath{#2}{#1}\ }
\newcommand{\wref}[2][ ]{\ \eqref{#2}{#1}\ }
\newcommand{\hsp}{\hspace{10 mm}}

\newcommand{\hsm}{\hspace{2 mm}}

\newcommand{\vsm}{\vspace{2 mm}}

\newcommand{\xra}[1]{\xrightarrow{#1}}

\newcommand{\Aut}{\operatorname{Aut}}
\newcommand{\AG}{\Aut(\Gamma)}
\newcommand{\clos}{\operatorname{cl}}

\newcommand{\cspace}{configuration space}

\newcommand{\Conf}{\operatorname{Conf}}

\newcommand{\Emb}[2]{\operatorname{Emb}^{#1}({#2})}

\newcommand{\eq}{\operatorname{eq}}
\newcommand{\Euc}[1]{\operatorname{Euc}^{#1}}
\newcommand{\Eucp}[1]{\operatorname{Euc}^{#1}_{+}}

\newcommand{\Id}{\operatorname{Id}}

\newcommand{\new}{\operatorname{new}}

\newcommand{\open}{\operatorname{op}}

\newcommand{\Or}[1]{\operatorname{O}({#1})}
\newcommand{\PSL}[1]{\operatorname{PSL}\sb{#1}(\R)}
\newcommand{\SO}[1]{\operatorname{SO}({#1})}

\newcommand{\be}{\mathbf{e}}
\newcommand{\ve}{\vec{\be}}

\newcommand{\vel}{\vec{\ell}}
\newcommand{\wel}{\widehat{\ell}}

\newcommand{\bss}{\mathbf{s}}
\newcommand{\vvs}{\vec{\bss}}
\newcommand{\bte}{\mathbf{t}}
\newcommand{\vt}{\vec{\bte}}
\newcommand{\vth}{\vec{\theta}}
\newcommand{\bu}{\mathbf{u}}
\newcommand{\hu}{\widehat{\bu}}
\newcommand{\bv}{\mathbf{v}}
\newcommand{\hv}{\widehat{\bv}}
\newcommand{\bw}{\mathbf{w}}
\newcommand{\hw}{\widehat{\bw}}
\newcommand{\bx}{\mathbf{x}}
\newcommand{\by}{\mathbf{y}}
\newcommand{\bz}{\mathbf{z}}
\newcommand{\RP}[1]{{\mathbf P}\RR{#1}}
\newcommand{\R}{{\mathbb R}}
\newcommand{\RR}[1]{\R\sp{#1}}
\newcommand{\hR}{\widehat{\R}}
\newcommand{\bT}[1]{{\mathbb T}\sp{#1}}
\newcommand{\Gn}[1]{\Gamma\sb{#1}}
\newcommand{\Gen}[1]{\Gamma\sp{\eq}\sb{#1}}

\newcommand{\Gnl}[1]{\Gn{#1}\sp{\open}}
\newcommand{\Gkl}{\Gnl{k}}
\newcommand{\Gkc}[1]{\Gn{#1}\sp{\clos}}
\newcommand{\TG}{\T\sp{\Gamma}}
\newcommand{\vv}{\vec{\bv}}

\newcommand{\hx}{\hat{\bx}}
\newcommand{\wx}{\tilde{\bx}}

\newcommand{\bt}{\mathbf{t}}
\newcommand{\hy}{\hat{\by}}
\newcommand{\wy}{\tilde{\by}}
\newcommand{\hz}{\hat{\bz}}
\newcommand{\wz}{\tilde{\bz}}
\newcommand{\bze}{\mathbf{0}}
\newcommand{\vz}{\vec{\bze}}

\newcommand{\halpha}{\widehat{\alpha}}
\newcommand{\hbeta}{\widehat{\beta}}
\newcommand{\C}{{\EuScript C}}
\newcommand{\Cs}{\C\sb{\ast}}
\newcommand{\Cso}{\C\sb{\ast}\sp{o}}
\newcommand{\hE}{\widehat{E}}
\newcommand{\wE}{\widetilde{E}}
\newcommand{\hF}{\widehat{F}}
\newcommand{\wF}{\widetilde{F}}
\newcommand{\CsG}{\Cs(\Gamma)}
\newcommand{\CsoG}{\Cso(\Gamma)}
\newcommand{\hC}{\widehat{\C}}
\newcommand{\wC}{\widetilde{\C}}
\newcommand{\hCo}{\widehat{\C}\sp{o}}
\newcommand{\wCo}{\widetilde{\C}\sp{o}}
\newcommand{\wSC}{\widetilde{\Sc\C}}
\newcommand{\wSG}{\wSC(\Gamma)}
\newcommand{\wSGn}[1]{\wSC(\Gen{#1})}
\newcommand{\hSC}{\widehat{\Sc\C}}

\newcommand{\CG}{\C(\Gamma)}
\newcommand{\hCG}{\hC(\Gamma)}
\newcommand{\wCG}{\wC(\Gamma)}
\newcommand{\wCGn}[1]{\wC(\Gen{#1})}
\newcommand{\hCoG}{\hCo(\Gamma)}
\newcommand{\wCoG}{\wCo(\Gamma)}

\newcommand{\SG}{\Sc\C(\Gamma)}
\newcommand{\SsG}{\Sc\CsG}
\newcommand{\hSG}{\hSC(\Gamma)}

\newcommand{\Fc}{{\mathcal F}}

\newcommand{\cM}{{\mathcal M}}
\newcommand{\Pc}{{\mathcal P}}
\newcommand{\Sc}{{\mathcal S}}

\newcommand{\Vc}{{\mathcal V}}

\newcommand{\T}{{\EuScript T}}

\newcommand{\mX}{\mathscr{X}}

\newcommand{\CX}{C\sb{\mX}}

\newcommand{\vare}{\varepsilon}

\begin{document}

\title[Symmetric configuration spaces of linkages]
{Symmetric configuration spaces of linkages}

\author[D. Blanc]{David Blanc}
\address{David Blanc \\ Department of Mathematics \\ Haifa University \\ Haifa 3498838, Israel}
\email{blanc@math.haifa.ac.il}

\author[N. Shvalb]{Nir Shvalb}
\address{Nir Shvalb \\ Department of Mechanical Engineering\\ Ariel University \\ Ariel 40700, Israel}
\email{nirsh@ariel.ac.il}


\begin{abstract}
A \emph{configuration} of a linkage $\Gamma$ is a
possible positioning of $\Gamma$ in \w[,]{\RR{d}} and the
collection of all such forms the \cspace\ \w{\CG} of $\Gamma$.
We here introduce the notion of the \emph{symmetric \cspace} of a linkage,
in which we identify configurations which are geometrically indistinguishable.
We show that the symmetric \cspace\ of a planar polygon has a regular cell structure,
provide some principles for calculating this structure, and give a complete description
of the symmetric \cspace\ of all quadrilaterals and of the equilateral pentagon.\end{abstract}

\maketitle

%
%
\section{Introduction}
\label{cint}
The mathematical theory of robotics is based on the notion of a
mechanism consisting of links connected by flexible joints. More precisely,
a \emph{linkage} $\Gamma$ is a metric graph, with edges (of fixed lengths)
corresponding to the links, and vertices corresponding to the joints. See
\cite{MerlP,SeliG,TsaiR} and \cite{FarbT} for surveys of the mechanical
and topological aspects, respectively.

A central tool in studying such a linkage is its \cspace\ \w[,]{\CG}
a topological space whose points correspond to possible positionings of $\Gamma$
in the ambient Euclidean space \w[.]{\RR{d}} These spaces are useful for
understanding actuations, motion planning, and singular configurations of
the mechanisms (see, e.g., \cite{FGranS,KTezuT,KTsukC,MTrinG,BBSShvaT,SSBlaCA});
in recent years, the related notion of topological complexity has been a topic of
much research (see \cite{FarbTC} and
\cite{BGRTamaH,BReciT,BKaluET,DCoheTC,DavisTC,FPaveS,MWuTC}).

Observe that in the standard construction of \w{\CG} we
distinguish between configurations which are functionally
equivalent though formally distinct: thus if $\Gamma$ consists of
a fixed platform with two identical free arms \w{ABC} and
\w[,]{ADE} the two positions shown in Figure \ref{ftwoarm} are
considered distinct configurations, but are functionally the same.
Thus it makes sense to consider a version of the configuration
space in which they are identified.

%
%
\begin{figure}[ht]
\centering{\includegraphics[width=0.5\textwidth]{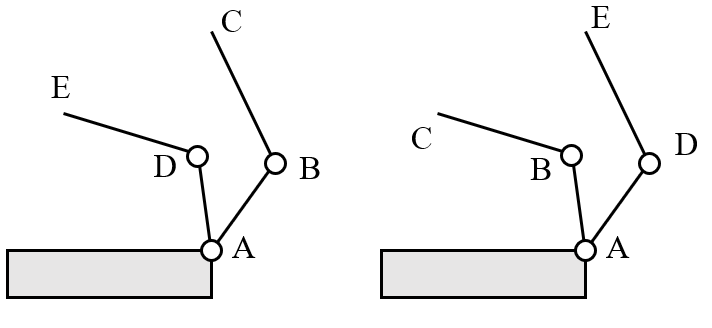}}
\caption{Two arm mechanism}
\label{ftwoarm}
\end{figure}

For this purpose, we introduce the notion of the \emph{symmetric
\cspace} of a linkage $\Gamma$, in which points of the usual
\cspace\ \w{\CG} are identified if they differ by an automorphism
of $\Gamma$ \wh which we can think of as a relabelling of vertices of $\Gamma$
which does not change its geometric relations (i.e., which vertices are connected by an
edge, and the length of this edge). One should note that there are
two useful versions of the \cspace\ of a linkage, \emph{fully
reduced} and \emph{reduced} (depending on whether we divide the
set of embeddings by all Euclidean isometries of the ambient space
\w[,]{\RR{d}} or only by the orientation-preserving ones \wh see \S
\ref{srestr} below). There are also two corresponding types of
symmetric \cspace s.

Although to the best of our knowledge, the concept of the symmetric \cspace\ of
a linkage has not appeared in the mathematical or engineering literature, it has an
obvious intuitive meaning: in real life, mechanisms do not have natural labellings
of their joints, and for practical purposes, the two arms of Figure \ref{ftwoarm}
are indistinguishable. Of course, if each arm is used to grasp a different object,
the distinction is important, which is why the usual notion of a \cspace\ is more
generally applicable.  However, for motion planning for the two-arm mechanism from
rest, the symmetric \cspace\ is the more economical version to use.

Intuitively, each point in the symmetric \cspace\ can be thought of as instructions
for specifying a rigid configuration of a real-life linkage in the ambient space
\w{\RR{d}} \wb[,]{d=2,3} without labelling links or joints which are indistinguishable
in the abstract mechanism $\Gamma$.

Another possible application is to molecules with structural
symmetries, whose reduced \cspace\ represents the mutual positions
of their constituent atoms in space (the fully reduced \cspace\
does not distinguish between the two chiralities, if they exist).
See, e.g., \cite[Ch.\ III]{HSharI}.

Our main results in this paper are concerned with the planar configurations of
the closed chain with $n$ links. On a theoretical level, we
provide a systematic approach to describing an equivariant cell
structure on the \cspace\ of a closed chain under action of the
group \w{\AG} of automorphisms of the linkage:

\begin{thma}
  The reduced \cspace\ of an $n$-gon in the plane has a regular
  \wwd{\AG}equivariant cell structure, subordinate to the standard regular cell
  structure, and similarly for the fully reduced \cspace.
\end{thma}
\noindent See Theorem \ref{ttriang} and Corollary \ref{ctriang} below.
From the equivariant cell structure for these two types of \cspace\  we can then derive
directly an induced (ordinary) cell structure for the two types of symmetric \cspace.

The symmetric cells of the \cspace\ \wwh that is, those
fixed under various subgroups of \w{H\leq\AG} \wwh play a central role in describing the
cell structure of two types of symmetric \cspace s, and we show:

\begin{thmb}
The fixed point set of the fully reduced \cspace\ of a planar
polygon $\Gamma$ under a subgroup $H$ of \w{\AG} is a disjoint
union of components (indexed by the discrete set of configurations
fixed under \w{\AG} itself), each of which fibers successively
over intervals or tori.
\end{thmb}
\noindent See Propositions \ref{lreflsym} and \ref{lsymmetcellcy} and
Theorem \ref{lsymmetcelldy} for a more precise description.

The remainder of the paper is devoted to two specific calculations.  We show:

\begin{thmc}
The reduced symmetric \cspace\ of a planar quadrilateral is
homeomorphic to a closed interval, a circle, a wedge of a circle and a segment,
or a circle with its diameter.
\end{thmc}
See Theorem \ref{tscsq} below, where the cases in which each value obtains are
described in full.

\begin{propd}
  The fully reduced symmetric \cspace\ of the planar equilateral pentagon
  is homeomorphic to a closed disc.
\end{propd}
See Proposition \ref{pscspent} below.

\begin{subsection}{Organization and main results}
\label{sorg}
In Section \ref{cbcspaces} we review the main notions needed to define the
various types of \cspace s of a mechanism, and introduce the corresponding versions of
symmetric \cspace s. Simple examples of symmetric configuration spaces are
given in Section \ref{cescspaces}, including a full description of planar quadrilaterals
(with the details appearing in Appendix \ref{cquad}). In Section \ref{cgppoly} we recall
some general facts about the \cspace s of planar polygons in general, including their
cell structure. Section \ref{capp} is devoted to the automorphisms of planar polygons
$\Gamma$, culminating in Theorem \ref{ttriang}. Section \ref{cscpp} discusses symmetric
configurations for planar polygons \wh that is, the fixed-point sets of the \cspace\ under
various subgroups $H$ of \w[.]{\AG} Section \ref{cfch} shows how these general results
may be applied to obtain an equivariant triangulation of one non-equivariant cell,
in the case where $\Gamma$ is the equilateral hexagon.
Finally, Section \ref{ceqpent} provides a full description of the fully reduced
symmetric \cspace\ of the equilateral pentagon.
\end{subsection}

%
%
\section{Configuration spaces}
\label{cbcspaces}

We first recall some general background material on the
construction and basic properties of \cspace s. This also serves
to fix notation, which is not always consistent in the literature.

\begin{definition}\label{dcspace}
Consider an abstract graph \w{\TG} with vertices $V$ and edges
\w{E\subseteq V\sp{2}} (with no loops or parallel edges, but not necessarily connected).
A \emph{linkage} (or \emph{mechanism}) $\Gamma$ of type
\w{\TG} is determined by a function \w{\ell:E\to\RR{}\sb{+}}
specifying the length \w{\ell\sb{i}} of each edge \w{e\sb{i}}
in \w{E=\{e\sb{i}=(u\sb{i},v\sb{i})\}\sb{i=1}\sp{k}} (subject to the triangle
inequality as needed).  We write
\w{\vel:=(\ell\sb{1},\dotsc,\ell\sb{k})\in\RR{E}} for the vector of lengths.

The \emph{\cspace} of the linkage $\Gamma$ is the metric subspace
\w{\CG:=\lambda^{-1}(\vel)} of \w{(\RR{d})\sp{V}}
(a real algebraic variety), where the map \w{\lambda:(\RR{d})\sp{V}\to\RR{E}}
is given by \w[.]{\lambda(u\sb{i},v\sb{i}):=\|\varphi(u\sb{i})-\varphi(v\sb{i})\|}
A point \w{\bx\in\CG} is called a \emph{configuration} of $\Gamma$. Note that
\w{\CG} is a subspace of the space \w{\Emb{d}{\TG}\subseteq(\RR{d})\sp{V}} of
\emph{embeddings} of $V$ in \w{\RR{d}} (without collisions).
\end{definition}

\begin{subsection}{Isometries of \cspace s}
\label{srestr}
The group \w{\Euc{d}} of isometries of the Euclidean space
\w{\RR{d}} acts on the space \w[.]{\CG} Taking this action into account allows us to
reduce the dimension of \w{\CG} without losing any interesting information,
as follows:

If we choose a fixed vertex \w{v\sb{0}} of $\Gamma$ as its \emph{base-point},
the action of the translation subgroup \w{T\cong\RR{d}} of \w{\Euc{d}}
on \w{\bx(v\sb{0})} \emph{is} free, so its action on \w{\CG} is free, too, and
we might reduce the degrees of freedom of \w{\CG} by considering its quotient
under this action.

However, such a choice will not fit in with our notion of
symmetries, so for our purposes it is more convenient to think of the coordinate
frame with the barycenter \w{B(\bx)} of a given configuration \w{\bx\in\CG}
at the origin.
We therefore define the \emph{pointed \cspace} for $\Gamma$ to be the quotient space
\w{\CsG:=\CG/T} under translations. Thus \w[,]{\CG\cong\CsG\times\RR{d}} and a pointed
configuration (i.e., an element of \w[)]{\CsG} is simply an ordinary
configuration expressed in terms of a coordinate frame for \w{\RR{d}}
with the barycenter at the origin.  Essentially, this means replacing the Euclidean ambient
space \w{\RR{d}} by the corresponding affine space, equipped with  a
chosen direction for each axis.

If we divide \w{\CG} by the action of the group \w{\Eucp{d}} of orientation-preserving
isometries of the ambient space \w[,]{\RR{d}} we obtain the \emph{reduced \cspace}
of $\Gamma$, denoted by \w[.]{\hCG} When $\Gamma$ has a rigid ``base platform'' $P$
of dimension \www[,]{\geq d-1} the action of \w{\Eucp{d}} is free.
For example, if \w[,]{d=2} we may fix a vertex \w{v\sb{0}} and a link $\vv$ in
$\Gamma$ starting at \w[,]{v\sb{0}}  and let \w{p:\CsG\to S\sp{d-1}} assign to a
configuration $\Vc$ the direction of $\vv$. The fiber of $p$ \w{\ve\sb{1}\in S\sp{d-1}}
is \w[.]{\hCG}

Dividing \w{\CG} by the full group \w{\Euc{d}} of all isometries of \w{\RR{d}}
we obtain the \emph{fully reduced \cspace} \w{\wCG} of $\Gamma$.
Note that any configuration whose image is contained in a line is fixed under
reflections in that line, so the action of \w{\Euc{d}} may not be free.
Thus \w{\CG} is not generally isomorphic to \w[.]{\wCG\times\Euc{d}}
Nevertheless, \w{\wCG} is the most economical model for most linkages $\Gamma$.
\end{subsection}

\begin{subsection}{Normalization}\label{snor}
Note that the image of a configuration \w{\bx\in\CG} is the same
as the image of \w{\bx\circ f} for any automorphism \w{f:V\to V}
of the linkage $\Gamma$ (that is, a relabelling of the vertices
preserving adjacency and edge length). This does not mean
that we have an isometry of \w{\RR{d}} taking $\bx$ to
\w{\by:=\bx\circ f} \wwh e.g., when \w{\TG} is a bouquet of
circles (closed chains), we may reflect one of them, leaving the
rest in place (see also Figure \ref{ftwoarm}).

However, when $\Gamma$ has a rigid ``base platform'' $P$ of dimension $\geq d-1$,
as above, we can identify \w{\CG} with \w[:]{\hCG\times\Eucp{d}} that is, every
configuration $\bx$ in \w{\CG} can be \emph{normalized} uniquely to a reduced
configuration \w[,]{\hx\in\hCG} by placing the platform in a standard direction and
moving the barycenter of $\bx$ to the origin. This will be denoted by
\w[,]{\hx=N(\bx)=T\sb{\bx}(\bx)} where the specific transformation
\w{T\sb{\bx}\in\Eucp{d}} used to normalize $\bx$ may not depend
continuously on $\bx$, but \w{N:\CG\to\hCG} \emph{is} continuous, thus providing
a canonical section for the quotient map \w[.]{q:\hCG\to\CG}

Any polygon $\Gamma$ in the plane always has such a rigid platform, so when $\Gamma$
is  equilateral, a reduced configuration $\hx$ is completely determined by the
orientation (that is, a cyclic ordering of the vertices
\w[)]{A\sb{1},\dotsc,A\sb{n},A\sb{1}}  and the sequence of angles
\w{(\phi\sb{1},\dotsc,\phi\sb{n})} at each vertex \w[,]{A\sb{i}}  measured counter
clockwise from \w{\overrightarrow{A\sb{i}A\sb{i+1}}} to
\w[.]{\overrightarrow{A\sb{i}A\sb{i-1}}}
The automorphism $f$ is simply a permutation $\sigma$ on \w{\{1,\dotsc, n\}} preserving
adjacency \wh that is, a cyclic shift, with or without a reverse of orientation.
If the orientation is preserved, \w[,]{N(\phi\sb{1},\dotsc,\phi\sb{n})=
  (\phi\sb{\sigma\sp{-1}(1)},\dotsc,\phi\sb{\sigma\sp{-1}(n)})}
  while if $\sigma$ reverses orientation, then  \w[,]{N(\phi\sb{1},\dotsc,\phi\sb{n})=
(-\phi\sb{\sigma\sp{-1}(1)},\dotsc,-\phi\sb{\sigma\sp{-1}(n)})} since the angles should
now be measured in the reverse direction.
\end{subsection}

\begin{example}\label{egopench}
  When \w{\Gamma=\Gkl} is an open chain of length $k$ (see Figure \ref{fmidref}
  below), \w[.]{\hCG\cong (S\sp{d-1})\sp{k-1}} If \w[,]{d=2} \w{\hC(\Gkl)}
is a \wwb{k-1}torus, parameterized by
\w[.]{(\theta\sb{1},\dotsc,\theta\sb{k-1})} The fully reduced
\cspace\ \w{\wC(\Gkl)} may be identified with a subspace of
\w{\hC(\Gkl)} defined as follows:

For \w[,]{k=2} \w[,]{\wC(\Gnl{2})=[0,\pi]\subseteq S\sp{1}} and we write
\w{\wC(\Gnl{2})'=[\pi,2\pi]} for the version of the fully reduced \cspace\
in which we require the first edge not on the $x$-axis to point downwards.

We may then define \w{\wC(\Gkl)} by induction on \w{k\geq 2} to be the subspace of
\w{(S\sp{1})\sp{k-1}} given by:
$$
(0,\pi)\times (S\sp{1})\sp{k-2}~\cup~
\{0\}\times\wC(\Gnl{k-1})~\cup~\{\pi\}\times\wC(\Gnl{k-1})'~.
$$

Thus for \w{k=3} we obtain a cylinder with each boundary component omitting half
a circle (opposite halves at either end).
\end{example}

\begin{remark}\label{rccs}
The \cspace s studied in this paper are mathematical models, which take
into account only the locations of the vertices of $\Gamma$,
disregarding possible  intersections of the edges. In the plane, there
is some justification for this, since we can allow one link to slide
over another. This is why this model is commonly used
(cf. \cite{FarbT,KMillM}; but see \cite{CDRoteS}). However, in
\w{\RR{3}} the model is not very realistic, since it disregards the
fact that rigid rods cannot pass through each other.

Note that \w{\Emb{3}{\TG}} has a dense open subspace \w{U(\TG)} consisting of
those embeddings of $V$ which induce an embedding of the full graph
(including its edges). Similarly, \w{\CG} has a dense open subspace
\w[.]{U(\Gamma):=\Emb{3}{\TG}\cap U(\TG)}
In a more realistic treatment of all configurations of $\Gamma$ in \w[,]{\RR{3}}
we must cut open \w{\CG} along the complement \w[,]{\CG\setminus U(\Gamma)} consisting
of configurations with collisions. The precise description of a ``realistic'' \cspace\
\w{\Conf(\TG)} is quite complicated, even at the combinatorial level, which is why
we work here with \w[,]{\Emb{d}{\TG}} \w[,]{\CG} and \w{\CsG} as defined in
\S \ref{dcspace}-\ref{srestr}.  We observe that even such a model \w{\Conf(\TG)}
is not completely realistic, in that it disregards the thickness of
the rigid rods. See \cite{BShvaC} for a fuller treatment of this issue.
\end{remark}

\begin{subsection}{Symmetric configurations}
\label{ssc}
When a mechanism $\Gamma$ has internal symmetries, the various
flavors of \cspace s described above may be
unnecessarily complicated: if we take into account the labelling of
the vertices, the two configurations in Figure \ref{ftwoarm} are not
equivalent even in the fully reduced \cspace\ for such a linkage,
though they may be the same from a practical point of view.
To overcome this discrepancy, consider the following notions:

A graph \w{\TG} as above has a discrete group  \w{\Aut(\TG)} of
  \emph{graph automorphisms}: the subgroup of the permutations
  \w{f:V\to V} on the vertex set $V$ which preserve the (undirected)
  edge relation. A mechanism $\Gamma$ with length function
  \w{\ell:E\to\RR{}\sb{+}} has a \emph{linkage automorphism} group
  \w[,]{\AG\subseteq\Aut(\TG)} consisting of those graph automorphisms \w{f:V\to V}
  which preserve lengths.
  The group \w{\AG} naturally acts on \w{\CG} on the right by precomposition:
  \w{\bx\mapsto \bx\circ f} (this means that we are simply \emph{relabelling} the
  vertices of the given geometric configuration $\bx$),
  and the quotient space \w{\SG:=\CG/\AG} is called the
\emph{full symmetric configuration space} of $\Gamma$. This action is not generally free.

Since the action of \w{\AG} preserves the barycenter \w{B(\bx)} of the configuration,
we define the \emph{pointed symmetric configuration space} of $\Gamma$ to be  the
subspace \w{\SsG} of \w{\SG} consisting of those equivalence classes \w{[\bx]}
with \w{B(\bx)} at the origin.

However, translating by \w{B(\bx)} yields a canonical isomorphism
\begin{equation}\label{eqpscs}
\SsG~\cong~\CsG/\AG~:=~\AG\sp{\open}\backslash\CG/\RR{d}~.
\end{equation}
\noindent (since the two actions commute). This suggests two further definitions:

The \emph{reduced symmetric configuration space} of $\Gamma$ is the quotient
\begin{equation}\label{eqrscs}
  \hSG:=\hCG/\AG:=\AG\sp{\open}\backslash\CG/\Eucp{d}\cong
  \AG\sp{\open}\backslash\CsG/\SO{d}~,
\end{equation}
\noindent  while the
\emph{fully reduced symmetric configuration space} of $\Gamma$ is defined to be
\begin{equation}\label{eqfrscs}
  \wSG~;=~\wCG/\AG~:=~\AG\sp{\open}\backslash\CG/\Euc{d}~\cong~
  \AG\sp{\open}\backslash\CsG/\Or{d}~,
\end{equation}
\noindent Neither is canonically describable as a subspace of \w[.]{\SG}
\end{subsection}

\begin{subsection}{Symmetries and normalization}\label{ssymnor}
As noted in \S \ref{snor} above, when $\Gamma$ has a rigid base platform $P$,
every configuration $\bx$ in \w{\CG} can be \emph{normalized} to a reduced configuration
\w{\hx=N(\bx)=T\sb{\bx}(\bx)} in \w[;]{\hCG} in particular, this is true
when $\Gamma$ is an equilateral polygon in the plane.

Under these assumptions, \w{\AG} acts not only on \w[,]{\CG} but also on \w[,]{\hCG}
by sending $\hx$ to \w[,]{N(\hx\circ f)} where we think of \w{\hCG} as a subspace of
\w{\CG} (but precomposition with \w{f\in\AG} may take us out of \w[).]{\hCG} We then have
\begin{equation}\label{eqshcg}
\SG~\cong~\hSG\times\Eucp{d}~:=~(\hCG/\AG)\times\Eucp{d}
\end{equation}
\noindent (see \wref[),]{eqrscs} so that
\begin{equation}\label{eqshcgo}
\SsG~\cong~\hSG\times\SO{d}~:=~(\hCG/\AG)\times\SO{d}~.
\end{equation}
\noindent (see \wref[).]{eqpscs}
\end{subsection}

\begin{remark}\label{rfullyred}
  Assume that the linkage $\Gamma$ has a rigid base platform $P$ of dimension
  $\geq d-1$ \ as
above. If no configuration $\bx$ of $\Gamma$ has its image fully
contained in an affine subspace of \w{\RR{d}} of dimension
\w[,]{d-1} we can always choose the normalization $\hx$ of $\bx$
to be ``positively oriented''. This means that we have a canonical
section \w[,]{\wCG\to\CG} so \w{\CG\cong\wCG\times\Euc{d}} and
thus \w[.]{\SG\cong\wSG\times\Euc{d}}

This will happen, for instance, if \w{d=2} and \w{\Gamma=\Gen{2k+1}} is an equilateral
polygon with an odd number of links.
\end{remark}

%
%
\section{Examples of symmetric configuration spaces}
\label{cescspaces}

We now describe a few simple examples of symmetric configuration spaces.
First, note the following:

\begin{remark}\label{rorder}
  When the graph \w{\TG} is a chain (either open or closed), any linkage $\Gamma$
  of type \w{\TG} is determined by the sequence \w{\vel:=(\ell\sb{1},\dotsc,\ell\sb{k})}
  of lengths of the consecutive links, and a pointed configuration $\bx$ for $\Gamma$ is
  thus completely determined by a sequence of vectors \w{(\vv\sb{1},\dotsc,\vv\sb{k})} in
  \w{\RR{d}}
  (with \w{\|\vv\sb{i}\|=\ell\sb{i}} for \w[),]{i=1,\dotsc,k} subject to the additional
  constraint that \w{\sum\sb{i=1}\sp{k}\,\vv\sb{i}=\vz} in the case of a closed chain
  (see \cite[\S 1.3]{FarbT}).

  We thus see that if we re-order the sequence \w[,]{\vel:=(\ell\sb{1},\dotsc,\ell\sb{k})}
  the new mechanism \w{\Gamma'} will have a canonically isomorphic configuration space.
  However, the automorphisms of $\Gamma$ need have no direct relations with those of
  \w[,]{\Gamma'} so the resulting symmetric configuration spaces may differ.
\end{remark}

\begin{subsection}{Open chains}
\label{sopench}
When \w{\TG} is an open chain of length $k$, there can be at most one non-trivial
automorphism of the corresponding linkage $\Gamma$ \wh namely, the inversion \wh if
\w{\vel:=(\ell\sb{1},\dotsc,\ell\sb{k})} is symmetric (i.e., \w{\ell\sb{i}=\ell\sb{k+1-i}}
for all \w[).]{1\leq i\leq k}

\begin{enumerate}
\renewcommand{\labelenumi}{(\roman{enumi})}
\item If \w[,]{k=d=2} with vertices \w[,]{A,B,C} the reduced \cspace\ \w{\hCG}
  is \w[,]{S\sp{1}} with parameter $\alpha$ equal to the angle from \w{\vec{BA}} to
  \w[,]{\vec{BC}}
  taken in the positive direction (and the fully reduced \cspace\ \w{\wCG} is a
  half circle).
  The \ww{C\sb{2}}-action (for symmetric $\vel$)
  switches $A$ and $C$, and therefore takes $\alpha$ to \w[.]{2\pi-\alpha} Thus we have
  two fixed points \wb[,]{\alpha=0,\pi} and \w{\hSG} is an arc (with endpoints
  corresponding to the two fixed points).
\item If \w{k=3} and \w[,]{d=2} with vertices \w[,]{A,B,C,D} then \w{\hCG} is a torus
\w{S\sp{1}\times S\sp{1}} with parameters \w{\theta=\angle ABC} and \w{\phi=\angle BCD}
  measured as above. For symmetric $\vel$, the \wwd{C\sb{2}}action
  takes \w{(\theta,\phi)} to \w[,]{(2\pi-\phi,2\pi-\theta)} with fixed points on the
  anti-diagonal \w[.]{\Delta:=\{(\theta,\phi)~|~\phi+\theta=2\pi\}} Since the action
  is generated by reflection in $\Delta$, if we think of \w{\hCG} as usual as a square
  with opposite sides identified, in \w{\hSG} we first reflect in $\Delta$ to obtain
  a triangle \w{\triangle PQR} (with \w{PR} corresponding to $\Delta$), and then
  identify \w{\vec{PQ}} with \w{\vec{QR}} to obtain a projective plane with a disc
  removed along \w{PR} \wwh that is, \w{\hSG} is a M\"{o}bius band.
\item As noted in \S \ref{egopench}, in general \w[.]{\hCG\cong
(S\sp{d-1})\sp{k-1}} When \w{k=2m+2} and \w[,]{d=2} \w{\hCG} is
parameterized by
 \w[.]{(\theta\sb{1},\dotsc,\theta\sb{m},\alpha,\phi\sb{m},\dotsc,\phi\sb{1})}
In the symmetric case, the \wwd{C\sb{2}}action sends this to
\w[,]{(2\pi-\phi\sb{1},\dotsc,2\pi-\phi\sb{m},2\pi-\alpha,
  2\pi-\theta\sb{m},\dotsc,2\pi-\theta\sb{1})} so the fixed points constitute
an
$m$-dimensional subspace with \w[.]{\alpha=0,\pi} When $k$ is odd, we have no central
parameter $\alpha$.
\end{enumerate}
\end{subsection}

\begin{subsection}{Triangles}
\label{strian}
When \w{\TG} is a closed chain of length $3$, there is only one embedding $\bx$ of any
triangle $\Gamma$ in \w[,]{\RR{2}}  up to isometry, so \w[.]{\CG\cong\Euc{2}}
\begin{enumerate}
\renewcommand{\labelenumi}{(\roman{enumi})}
\item When $\Gamma$ is scalene, \w{\Aut(\Gamma)} is trivial, so \w[.]{\SG=\CG}
\item When $\Gamma$ is an isosceles triangle, \w[,]{\Aut(\Gamma)=C\sb{2}}
and its action on
  an embedding \w{\bx:\Gamma\to\RR{2}} is equivalent to reflection in the median, so
  $$
  \SG=\Euc{2}/C\sb{2}\cong\Eucp{2}=\RR{2}\ltimes S\sp{1}
  $$
\noindent and thus \w{\SsG\cong S\sp{1}} and \w{\hSG} is a single point.
\item When \w{\Gamma=\Gen{3}} is equilateral, \w{\Aut(\Gamma)=S\sb{3}} is the full
  symmetric group, so after dividing out by the reflection we have
  \w{\Aut(\Gamma)/\{\pm 1\}=A\sb{3}} (a cyclic group of order $3$), and thus
  \w[,]{\SG=\RR{2}\ltimes (S\sp{1}/A\sb{3})} with \w{\SsG\cong S\sp{1}} again.
\end{enumerate}
\end{subsection}

\begin{subsection}{Quadrilaterals}
\label{squad}
The usual configuration spaces of planar quadrilaterals are easy
to analyze in terms of the length vector
\w[.]{\vel=(\ell\sb{1},\ell\sb{2},\ell\sb{3},\ell\sb{4})} Note
that $\vel$ must satisfy certain inequalities in order to have a
non-empty (or non-trivial) configuration space \wh e.g., if
\w[,]{\ell\sb{1}>\ell\sb{2}} we must have
\w{\ell\sb{1}-\ell\sb{2}<\ell\sb{3}+\ell\sb{4}} (see \cite[Lemma
1.4]{FarbT})

It is easy to see that \w{\AG} is non-trivial exactly in the following four cases:

\begin{enumerate}
\renewcommand{\labelenumi}{(\arabic{enumi})}
\item An equilateral quadrilateral \w{\Gamma=\Gen{4}} with
  \w[,]{\ell\sb{1}=\ell\sb{2}=\ell\sb{3}=\ell\sb{4}} \wwh in which case \w{\AG=D\sb{4}}
  (the dihedral group of order $8$).
\item Two pairs of adjacent links have equal lengths:
  \w[,]{\ell\sb{1}=\ell\sb{2}>\ell\sb{3}=\ell\sb{4}}
  \wwh in which case \w[,]{\AG=C\sb{2}}  generated by the reflection exchanging the
  equal links.
\item Each pair of opposing links has equal lengths (distinct from each other)
  \wh in which case \w[,]{\AG=C\sb{2}\times C\sb{2}}
  generated by the two interchanges of opposites.
\item Two opposite links are equal (but not all four) \wh in which case \w[,]{\AG=C\sb{2}}
  generated by the reflection exchanging the equal links.
\end{enumerate}
\end{subsection}

\begin{subsection}{Parameterizing the configuration space of a quadrilateral}
\label{spcsp}
When \w{\Gamma=ABCD}  is a quadrilateral in the plane, for generic (non-equilateral)
$\vel$, any reduced configuration $\hx$ of $\Gamma$  is determined
by the angle \w{\phi=\angle BAD} (measured counter clockwise from \w{\vec{AD}} to
\w[),]{\vec{AB}}  together with the ``elbow up''-``elbow down'' position
$\vare$ of \w{BCD} \wwh that is, whether
the \w{\beta=\angle BCD} (measured counter clockwise from \w{\vec{CB}} to \w[)]{\vec{CD}}
    satisfies \w[,]{0<\beta<\pi} in which case
  \w[,]{\vare:=+1} or \w[,]{\pi<\beta<2\pi} in which case \w[.]{\vare:=-1} If \w{\beta=0} or
  \w[,]{\alpha=\pi} we say \w{ABC} is \emph{aligned} and set \w[.]{\vare:=0} Note that this last case
  is completely determined by the value of $\phi$ (but may never occur, depending on the
  length vector \w[.]{\vel}

  In the equilateral case, when \w{\phi=0=\alpha} (that is, the edge \w{AB} coincides
  with \w{AD} and \w{CB} coincides with \w[),]{CD} we need an additional parameter,
  namely, the angle
  $\theta$ between these two collapsed intervals. This suggests that the correct way to parameterize
  \w{\hCG} (for any planar $n$-polygon) is to embed it into the $n$-torus
  \w{\bT{n}=(S\sp{1})\sp{n}} by keeping track of the angles (oriented as above) at \emph{each}
  vertex. In this situation \w{\CG} is the subspace of \w{\bT{n}} determined by the requirement that
  the open chain reduced configuration $\hx$ defined by any \w{n-1} successive
  vertex angles close up and the new angle formed is equal to the remaining vertex
  angle parameter. This description is more wasteful than the above, but is better
  suited to discussing symmetries.
\end{subsection}

\begin{subsection}{The reduced \cspace\ of a quadrilateral}
\label{socsq}
  By Remark \ref{rorder}, for the ordinary reduced \cspace\ \w{\hCG} of a quadrilateral,
  we may assume \w[.]{\ell\sb{1}\geq\ell\sb{2}\geq\ell\sb{3}\geq\ell\sb{4}}

Using the method of \cite{MTrinG}, we have the following six cases, with the reduced
\cspace\ \w{\hCG} given in \cite[\S 1.3]{FarbT}:

\begin{enumerate}
\renewcommand{\labelenumi}{(\roman{enumi})}
\item \w{\ell\sb{2}<\ell\sb{1}<\ell\sb{2}+\ell\sb{3}+\ell\sb{4}}  and
  \w[,]{\ell\sb{2}+\ell\sb{3}<\ell\sb{1}+\ell\sb{4}} with \w[.]{\CG\cong S\sp{1}}

  Here \w{\AG\neq\{1\}} if and only if
  \w{\ell\sb{2}=\ell\sb{3}} or \w{\ell\sb{3}=\ell\sb{4}} (opposing) or
    any linkage for which \w[.]{\ell\sb{1}>\ell\sb{2}=\ell\sb{3}=\ell\sb{4}}
\item \w{\ell\sb{2}<\ell\sb{1}<\ell\sb{2}+\ell\sb{3}+\ell\sb{4}}  and
  \w[,]{\ell\sb{2}+\ell\sb{3}=\ell\sb{1}+\ell\sb{4}} with \w[.]{\CG\cong S\sp{1}\vee S\sp{1}}

  Here \w{\AG\neq\{1\}} if and only if \w{\ell\sb{2}=\ell\sb{3}} (opposing).
\item \w{\ell\sb{1}>\ell\sb{2}} and \w[,]{\ell\sb{2}+\ell\sb{3}>\ell\sb{1}+\ell\sb{4}}
  with \w[.]{\CG\cong S\sp{1}\amalg S\sp{1}}

   Again \w{\AG\neq\{1\}} if and only if \w{\ell\sb{2}=\ell\sb{3}} (opposing).
 \item \w[,]{\ell\sb{1}=\ell\sb{2}\geq\ell\sb{3}>\ell\sb{4}}
   with \w[.]{\CG\cong S\sp{1}\amalg S\sp{1}}

   Here \w{\AG\neq\{1\}} if and only if \w{\ell\sb{1}} is opposite \w[.]{\ell\sb{2}}
 \item \w[,]{\ell\sb{1}=\ell\sb{2}>\ell\sb{3}=\ell\sb{4}} with
   \w{\hCG} as in Figure \ref{fparallelogram}.

   Here \w{\AG=C\sb{2}} or \w{C\sb{2}\times C\sb{2}} in cases (2) or (3) respectively.
 \item \w[,]{\ell\sb{1}=\ell\sb{2}=\ell\sb{3}=\ell\sb{4}}
   with \w{\CG\cong S\sp{1}\vee S\sp{1}\vee S\sp{1}\vee S\sp{1}} and \w[.]{\AG=D\sb{4}}
\end{enumerate}
\end{subsection}

\begin{notation}\label{nisoctrap}
  For a quadrilateral \w{\Gamma=\Box ABCD} we write \w[,]{p:=AB} \w[,]{q:=BC} \w[,]{r:=CD}
  and \w{s:=DA} for the four values of $\vel$ in order, and assume without loss of
  generality that \w{s=\ell\sb{1}} is the longest.
\end{notation}

\begin{theorem}\label{tscsq}
In the notation of \S \ref{nisoctrap}, the symmetric \cspace\ \w{\hSG}
of any quadrilateral $\Gamma$ is:

\begin{enumerate}
\renewcommand{\labelenumi}{(\roman{enumi})}
\item A  closed interval, if \w[,]{s>p,q,r} \w[,]{p=r} \w[,]{s<p+q+r} and
  \w[;]{s+q>2p}
\item A circle, if \w[,]{s>p,q,r} \w[,]{p=r} \w[,]{s<p+q+r} and
  \w[,]{s+q\leq 2p}  or if \w[.]{s=q\geq p>r}
\item A wedge of a circle and a segment, if \w{s=q> p=r} (a parallelogram).
\item A circle with its diameter, if \w{s=p>q=r} (a deltoid).
\item A closed interval, if \w{s=p=q=r} (an equilateral quadrilateral).
\end{enumerate}
\end{theorem}

Since the proof of Theorem \ref{tscsq} requires checking many special cases,
it is relegated to Appendix \ref{cquad}.

\begin{remark}\label{rredfullsym}
The fully reduced symmetric \cspace\ \w{\wSG} is generally different from \w[.]{\hSG}
However, for the equilateral quadrilateral \w[,]{\Gamma=\Gen{4}} they are the identical,
in fact. This is because the fully reduced \cspace\ \w{\wCG} is the upper half of
the reduced \cspace\ \w{\hCG} (see Figure \ref{fsquare}), and \w{\wSG} is then obtained
from \w{\wCG} by further identifying the right and left hand sides.

Intuitively, this is because an abstract (i.e., unlabelled)
configuration for a rhombus has no distinguished orientation
(unlike a scalene quadrilateral, where all four angles are
generally distinct, so can be oriented in the positive direction).
\end{remark}

%
%
\section{General planar polygons}
\label{cgppoly}

When \w{\TG} is a polygon with more than four edges, the analysis we made above
(and in Appendix \ref{cquad}) becomes more complicated, and the number of cases too
large for a full description to be useful.
However, it may still be possible to say something about the cell structure of
the \cspace\ \w{\CG} for the various mechanisms $\Gamma$ of type \w[.]{\TG}
Moreover, the diffeomorphism type of \w{\CG} (as a manifold with singularities)
depends generically only on inequalities among sums of subsets of $\vel$
(see \cite[Theorem 1.1]{HRodrS}).

In this section, we recall a classical approach to such cell structures.

\begin{subsection}{Arrow diagrams}
\label{sparamcs}
The pointed \cspace\ \w{\CsG} of a closed $n$-chain
\w{\Gamma=\Gkc{n}} in \w[,]{\RR{2}} with length vector $\vel$, may
be parameterized by points \w{\vth} in the $n$-torus
\w[,]{\bT{n}=(S\sp{1})\sp{n}} where for a configuration
\w[,]{\bx\in\wCG\subset\Cs(\Gamma)}
\w{\vth(\bx)=(\theta\sb{1},\dotsc,\theta\sb{n})} is the vector of
angles \w{\theta\sb{i}} between \w{\bx(e\sb{i})} and the positive
$x$-axis. This encodes the one-to-one correspondence between a
configuration $\bx$ of $\Gamma$ and the \emph{arrow diagram}
obtained by moving all vectors \w{\bx(e\sb{i})} to the origin. See
Figure \ref{farrowd}, where (a) is a configuration of a pentagon,
and (b) is the corresponding arrow diagram.

%
%
\begin{figure}[ht]
\centering{\includegraphics[width=0.7\textwidth]{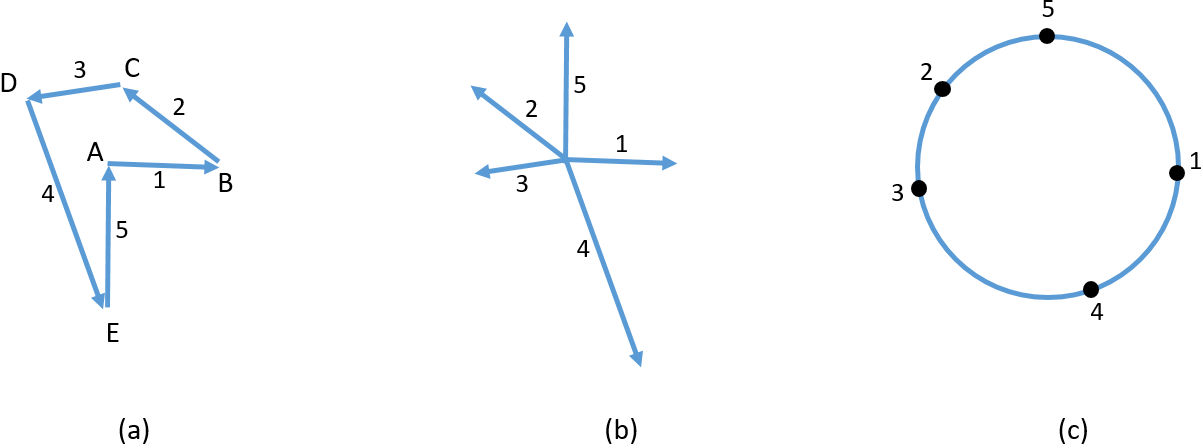}}
\caption{Arrow diagram}
\label{farrowd}
\end{figure}

Of course, not every value of \w{\vth\in(S\sp{1})\sp{n}} is allowed \wh
we have two constraints,
\begin{equation}\label{eqsumsin}
  \sum_{i=1}\sp{n}\ \ell\sb{i}\cos\theta\sb{i}~=~0\hsp\text{and}\hsp
  \sum_{i=1}\sp{n}\ \ell\sb{i}\sin\theta\sb{i}~=~0~,
\end{equation}
\noindent to ensure that the chain is closed.

The reduced configuration space is then \w[,]{\hCG=\CsG/\SO{2}}
with the circle \w{\SO{2}} acting by rotating a configuration
about the origin (as in \S \ref{srestr}). It may be parameterized
by \w[,]{\vth\in\bT{n-1}} since we fix \w[.]{\theta\sb{1}=0} In
order for $\bx$ to be fully reduced, we also require that the
first \w{\theta\sb{i}} \wb{i\geq 2} which is not $0$ or $\pi$ must
be \w[.]{<\pi}
\end{subsection}

\begin{subsection}{Cells in the torus}
\label{scelltor}
The torus \w{\bT{n}} is decomposed into $n$-dimensional cells by (the images of)
the hyperplanes in \w{\RR{n}} given by conditions of the form
\w{\theta\sb{i}=\theta\sb{i+1}}
or \w{\theta\sb{i}=\theta\sb{i+1}+\pi} for \w[.]{1\leq i\leq n} In any connected
component $E$ of the complement of these hyperplanes, we may use any
consecutive \w{n-3} of the parameters \w{(\theta\sb{1},\dotsc,\theta\sb{n})} as
local coordinates for \w{\hCG} (cf.\ \cite[Theorem 1.3]{FarbT}).
However, we have no control over the intersection of $E$ with \w[,]{\hCG} which may
no longer be connected, for example, so we follow the approach of \cite{KMillM},
in the formulation of \cite{PaninM}, to describe a regular cell structure on
\w[.]{\hCG}

For this purpose it is convenient to use the homeomorphism \w{\phi:S\sp{1}\to\hR}
given by \w[,]{\phi(\theta)=\tan(\theta/2)} where \w{\hR:=\RP{1}\cong\R\cup{\infty}}
is the real projective line. This defines a
coordinate-wise identification \w{\Phi:\bT{n}\to\hR\sp{n}} (the $n$-fold product),
with
\begin{equation}\label{eqtphi}
  \vt~=~(t\sb{1},\dotsc,t\sb{n})~=~\Phi(\vth)~:=~
  (\phi(\theta\sb{1}),\dotsc,\phi(\theta\sb{n}))~.
\end{equation}
Moreover, since \w[,]{\phi(0)=0} the image of \w{\hCG\subseteq\CsG\subseteq\bT{n}} under
$\Phi$ is contained in \w[.]{\{0\}\times\hR\sp{n-1}}

Note that the projective special linear group \w{\PSL{2}} acts by M\"{o}bius
transformations on $\hR$, and thus diagonally on \w[.]{\hR\sp{n}}
It turns out that the orbit space \w{\hR\sp{n}/\PSL{2}} is isomorphic to \w{\hCG} by
\cite[Theorem 4]{KMillM} (which is stated in terms of stable measures on
\w[,]{S\sp{1}} allowing one to conformally transform any arrow diagram $\vth$ into
one with vector sum at the origin).
\end{subsection}

\begin{subsection}{Cells for the \cspace\ of a polygon}
\label{scellcs}
As in \cite[\S 1]{PaninM}, we first note that for
\w{\Gamma=\Gkc{n}} an $n$-polygon (in \w[),]{\RR{2}} with length
vector $\vel$, we have a dense open subset \w{\CsoG} of \w{\CsG}
consisting of those arrow diagrams $\vth$ with all angles
distinct. To each such $\vth$ we associate a \emph{cyclic
ordering} \w{\alpha(\vth)} of the arrows, labelled by
\w[,]{\{1,\dotsc,n\}} on the circle: that is, a coset of the
symmetric group \w{S\sb{n}} modulo the left action of the subgroup
\w{C\sb{n}} generated by the cyclic permutation \w[.]{(2,3,\dotsc,
n,1)} This coset is obtained from the given labelling of the
arrows by the edges \w{(e\sb{1},\dotsc,e\sb{n})} of \w{\Gkc{n}} by
selecting an arbitrary starting point, and the action of
\w{C\sb{n}} corresponds to choosing a new starting point. See (c)
in Figure \ref{farrowd}.

Since this process respects the \wwd{\SO{2}}action on a pointed configuration
\w[,]{\bx\in\CsG} and thus on \w[,]{\vth(\bx)} the map
\w{\alpha:\CsoG\to S\sb{n}/C\sb{n}} descends to \w[,]{\halpha:\hCoG\to S\sb{n}/C\sb{n}}
where \w{\hCoG} is the corresponding dense open subset of the reduced \cspace\ \w[.]{\hCG}

The preimage under $\halpha$ of each coset \w{[\sigma]\in
S\sb{n}/C\sb{n}} is an open cell \w{\hE\sb{\sigma}} in \w[.]{\hCG}
Any two such cells are homeomorphic under a relabelling of the
arrows, so we may concentrate on the cell \w{\hE\sb{\Id}}
corresponding to the identity permutation: thus \w{\hE\sb{\Id}}
consists of strictly convex configurations of $\Gamma$ (in the
upper half plane).

To see that \w{\hE\sb{\Id}} is in fact a regular cell (that is, homeomorphic to a ball,
with a regular cell structure on its boundary), we proceed as follows
(see \cite[Lemma 1.2]{PaninM}):

Note that \w{\Phi:\bT{n}\to\hR\sp{n}} preserves the cyclic order of the coordinates:
i.e., \w{\theta\sb{i}\prec\theta\sb{j}\prec\theta\sb{k}} in the positive direction on
\w{S\sp{1}} if and only if \w{t\sb{i}\prec t\sb{j}\prec t\sb{k}} in $\hR$, and this
cyclic order on $\hR$ is preserved by M\"{o}bius transformations.
Moreover, if \w{\hE=\hE\sb{\sigma}} is any open cell in \w[,]{\hCG\subseteq\bT{n}}
any \w{\vth\in\hE} has pairwise distinct coordinates, so the same is true of
the set \w[.]{\hE':=\Phi(\hE)} By standard facts about M\"{o}bius transformations,
for each \w{\vt=\Phi(\vth)\in\hE'} there is a unique \w{\Pc\sb{\vt}\in\PSL{2}}
such that \w[,]{\Pc\sb{\vt}(t\sb{1})=\infty} \w[,]{\Pc\sb{\vt}(t\sb{2})=0} and
\w[.]{\Pc\sb{\vt}(t\sb{n})=1} It is given by:
\begin{equation}\label{eqmoebius}
w~=~\Pc(t)~:=~\frac{t-t\sb{2}}{t-t\sb{1}}\cdot\frac{t\sb{n}-t\sb{1}}{t\sb{n}-t\sb{2}}~,
\end{equation}
\noindent which simplifies when \w{t\sb{1}=0} to
\w[,]{\Pc(t)~:=~\frac{t\sb{n}(t-t\sb{2})}{t(t\sb{n}-t\sb{2})}} with inverse
\begin{equation}\label{eqmoebiusinv}
t~=~\Pc\sp{-1}(w)~:=~\frac{t\sb{2}t\sb{n}}{(t\sb{2}-t\sb{n})w+t\sb{n}}~.
\end{equation}
Moreover, the correspondence \w{\bt\mapsto\Pc\sb{\vt}} is continuous and one-to-one.

Thus we have a map \w[,]{\Pc:\hE'\to\hR\sp{n}} defined by
\begin{equation}\label{eqmoebiusp}
\Pc(\vt)~=~~(\infty,0,\Pc\sb{\vt}(t\sb{3}),\dotsc,\Pc\sb{\vt}(t\sb{n-1}),1)
\end{equation}
\noindent which is a bijection onto its image \w[.]{\hE''} Because \w{\Pc\sb{\vt}}
preserves the ordering of the coordinates in $\hR$, we see that for
\w[,]{\hE=\hE\sb{\Id}} \w{\hE''} is isomorphic to the affine cell
\begin{equation}\label{eqaffinecell}
\{\vvs=(s\sb{3},\dotsc,s\sb{n-1})\in\RR{n-3}~|\ 0<s\sb{3}<s\sb{4}<\dotsc<s\sb{n-1}<1~\}~.
\end{equation}
\noindent Since for any cyclic permutation $\sigma$, \w{\hE\sb{\sigma}} is isomorphic
to \w[,]{\hE\sb{\Id}} it is also an affine cell under the appropriate identifications.

The top cells in \w{\partial\hE} correspond to the
various cases where two adjacent arrows in $\vth$ coincide, in which case the
corresponding configuration \w{\bx=\bx(\vth)} is a strictly convex configuration for
a closed polygon with \w{n-1} edges, and with length vector \w{\vel'} such that
\w{\ell'\sb{i}=\ell\sb{i}} for
\w[,]{1\leq i\sb{0}} \w{\ell'\sb{i\sb{0}}=\ell\sb{i\sb{0}}+\ell\sb{i\sb{0}+1}} and
\w{\ell'\sb{i}=\ell\sb{i+1}} for \w[.]{i\sb{0}<i<n} Similarly by recursion for the
remaining cells in \w[.]{\partial\hE''}

In almost all cases these top cells correspond to the analogous boundary cells of
\w{\hE''} with equalities in \wref[.]{eqaffinecell} The cases where
\w[,]{\theta\sb{1}=\theta\sb{2}}
say, are treated by changing our choice of which coordinates of $\vt$ are sent to
\w{(\infty,0,1)} in \wref[.]{eqmoebiusp}
\end{subsection}

\begin{subsection}{Cells for the fully reduced \cspace}
\label{scfrcs}
For the fully reduced case, note that in \w{\CsoG} we can also associate to an
arrow diagram $\vth$ a coset \w{\beta(\vth)} of the
symmetric group \w{S\sb{n}} modulo the left action of the dihedral subgroup
\w[,]{D\sb{n}} which acts on a labelling of the arrows in $\vth$ by varying both the
starting point and the direction in which we proceed. Since \w[,]{C\sb{n}<D\sb{n}}
we have a surjection \w[,]{\pi:S\sb{n}/D\sb{n}\to S\sb{n}/C\sb{n}} with
\w[.]{\alpha=\pi\circ\beta} We call \w{\beta(\vth)} the \emph{dihedral ordering}
associated to $\vth$.

Again, this process respects the \wwd{\Or{2}}action on a pointed configuration
\w[,]{\bx\in\CsG} and thus on \w[,]{\vth(\bx)} so \w{\beta:\CsoG\to S\sb{n}/D\sb{n}}
induces \w[,]{\hbeta:\wCoG\to S\sb{n}/C\sb{n}} where \w{\wCoG} is the corresponding
dense open subset of the fully reduced \cspace\ \w[.]{\wCG}

Note that we have a trivial double covering map \w[,]{\delta\sp{o}:\hCoG\to\wCoG} although
the corresponding map \w{\delta:\hCG\to\wCG} has branch points at fully aligned
configurations of $\Gamma$ (if they exist, as for even equilateral polygons).

Thus the preimage under $\hbeta$ of each coset \w{[\sigma]\in S\sb{n}/D\sb{n}}
is again an open cell \w{\wE\sb{\sigma}} in \w[,]{\wCG} doubly covered by
\w[,]{\hE\sb{\sigma'}\amalg\hE\sb{\sigma''}} where
\w[.]{\{[\sigma'],[\sigma'']\}=\pi\sp{-1}[\sigma]} We must be more careful in analyzing
the boundary \w[,]{\partial\wE\sb{\sigma}} since \w{\delta:\hCG\to\wCG} may have
branch points there.
\end{subsection}

%
%
\section{Automorphisms of planar polygons}
\label{capp}

In Section \ref{cgppoly} we summarized briefly the standard
approach to describing the cell structure of the (fully) reduced
\cspace s of polygons in the plane. We now turn to a finer cell
structure needed to analyze the symmetric \cspace s. First, we
note a straightforward result about \w[,]{\AG} for closed chains
$\Gamma$:

\begin{definition}\label{dautograph}
  Let $\wel$ be the cyclic word in $n$ positive real numbers corresponding to the (ordered)
  length vector $\vel$. The maximal number \w{k\geq 1} of repeating segments $I$ into which
  $\wel$ can be divided will be called the \emph{order} of $\vel$ (so \w[).]{k|n}

If $\wel$ is symmetric (with respect to reversing the order from a certain starting point),
we call $\vel$ \emph{palindromic}. More generally, if after possibly omitting a single
length \w{\ell\sb{i}} from $\wel$ it becomes palindromic, we say that $\vel$ is
\emph{reflective}. Thus \w{\wel=(1,2,3,2,1)} is palindromic, while
\w{(1,2,3,2,1,4)} is reflective.
\end{definition}

\begin{lemma}\label{pautograph}
Given a length vector $\vel$ a closed chain \w{\Gamma=\Gkc{n}} of length $n$,
  the automorphism group of the linkage $\Gamma$ is
$$
\AG~\cong~\begin{cases}
  D\sb{k} & \text{if}~ \vel~\text{is reflective}\\
  C\sb{k} & \text{if}~ \vel~ \text{is not reflective,}
  \end{cases}
$$
\noindent where \w[,]{D\sb{n}} the dihedral group of order \w[,]{2n} is the group of
symmetries of the equilateral $n$-gon \w[,]{\Gen{n}} so in particular
\w{D\sb{2}:=C\sb{2}} and \w[.]{C\sb{1}=\{1\}}
\end{lemma}

\begin{proof}
  If \w{k\geq 2} and $\vel$ is non-reflective, rotation along the repeated segment
  $I$ generates \w[.]{\AG=C\sb{k}}

If $\vel$ is reflective, we have in addition a reflection in an axis which either
connects two opposite vertices $u$ and $v$ (if $n$ is even and $\vel$ is palindromic
starting at $u$); connects a vertex to the midpoint of an edge (if $n$ is odd
and $\vel$ is palindromic); or connects the midpoints of two opposite edges $e$ and $f$
(if $n$ is even and $\vel$ is palindromic after dropping $e$, say).
If also \w[,]{k\geq 2} we have a total of $k$ possible reflections of this form.
\end{proof}

\begin{subsection}{$\AG$-cells}
\label{sagcell}
The key to understanding the symmetric configuration space is using appropriate cells
for the description of the reduced \cspace\ \w[,]{\hCG} and its fully reduced version
\w[,]{\wCG} which take into account the \wwd{\AG}action.  More precisely, we decompose
\w{\hCG)} (or \w[)]{\wCG} into open \emph{free} $H$-cells for each subgroup $H$
of \w{\AG} \wwh that is, open cells of the form \w[,]{\be\sp{n}\times H} with the
$H$ acting on the second coordinate (see \cite{MayEHC}).

The highest dimensional cells in this decomposition (with
\w[)]{\dim(\CX)=n-3} will be those for which the action of \w{\AG}
is free (in the interior). This simply reflects the fact that
generic configurations have no symmetries, if \w{n\geq 5} (which
fails for the equilateral quadrilateral, as we see in Lemma
\ref{lsquare}). Note, however, that the action may take an open
cell \w{\hE\sb{\sigma}} to itself; to avoid this, we must
subdivide it into finer (open) cells which are permuted among
themselves by \w{\AG} (acting under relabelling combined with
normalization back to the reduced form).

These fine cells are determined by ``breaking the symmetry'' of the corresponding
arrow diagram \wh that is, imposing an additional (open) condition which determines a
unique ``canonical'' labelling of the vertices.
\end{subsection}

\begin{remark}\label{rlattice}
The subgroup lattice of \w{\AG} plays a central role in our analysis
of the equivariant cell structure. Note, however, that for certain subgroups $H$
of \w[,]{\AG} any configuration stabilized by $H$ are in fact invariant under a larger
subgroup.  Thus for example when \w{\Gamma=\Gen{6}} is the equilateral hexagon, the
three subgroups shaded in the lattice of subgroups of \w{\AG=D\sb{12}} in
Figure \ref{fsubgplat} will never appear as stabilizers.
%
%
\begin{figure}[ht]
\centering{\includegraphics[width=0.5\textwidth]{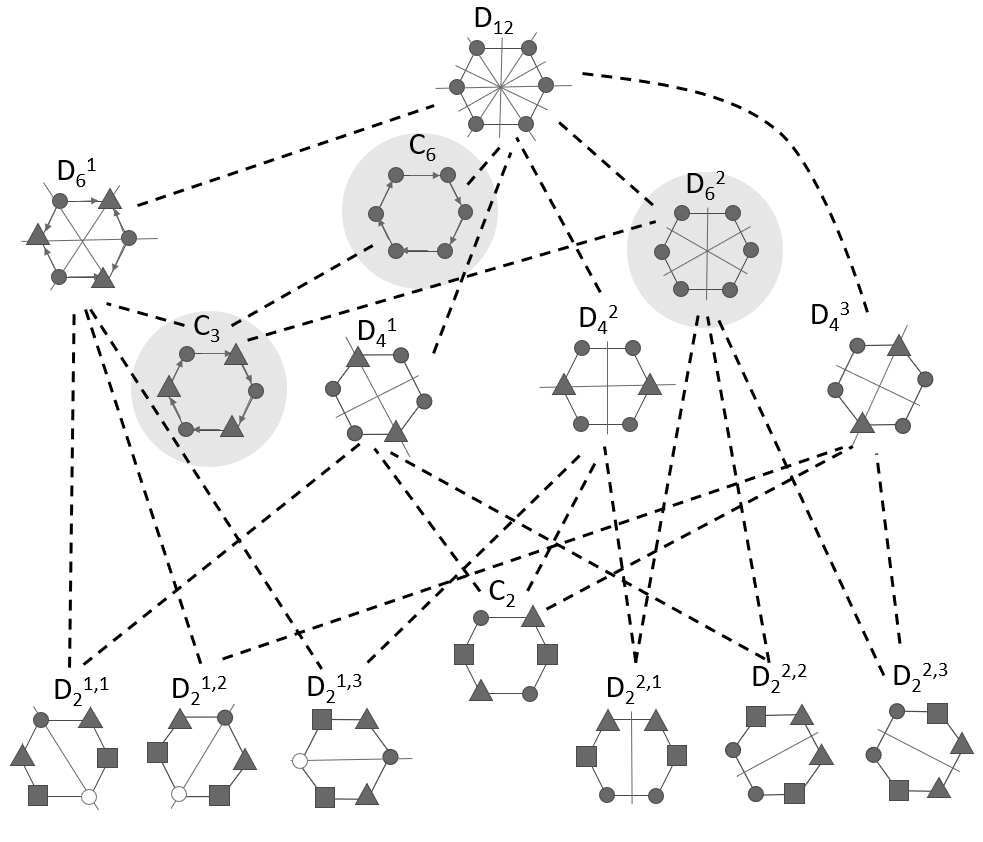}}
\caption{Subgroup lattice of $D\sb{12}$}
\label{fsubgplat}
\end{figure}
\end{remark}

\begin{definition}\label{dautsigma}
Consider an open cell \w{\wE\sb{\sigma}} for \w{\wCG} corresponding to a
cyclic (or dihedral) ordering \w{[\sigma]} in \w{S\sb{n}/C\sb{n}}
(or \w[)]{S\sb{n}/D\sb{n}} \wh see \S \ref{scfrcs}.

It is convenient to think of our geometric representation of the
abstract cyclic (or dihedral) ordering \w{[\sigma]\in
S\sb{n}/D\sb{n}} as a certain (fully) reduced configuration of a
mechanism \w{\Delta\sb{n}} consisting of $n$ unit vectors
emanating from the same point in the plane. Since we are only
interested in the cyclic ordering, we choose a normalized
representative \w{c\sb{\sigma}} in \w{\hC(\Delta\sb{n})}
(respectively,  \w[)]{\wC(\Delta\sb{n})} for \w{[\sigma]} in which
the arrow heads lie at the cyclotomic points
\w{c\sb{k}=\frac{2(k-1)\pi}{n}} \wb[,]{k=1,\dotsc,n} suitably
labelled, with \w[.]{\sigma(1)=1}

The action of \w{\psi\in\AG\subseteq D\sb{n}}
is by changing the labelling \w{\sigma(k)} of \w{c\sb{k}} to \w[,]{\psi(\sigma(k))}
and then renormalizing by applying a cyclic shift to \w{\psi\circ\sigma} to obtain a new
permutation \w{\sigma'} with \w{\sigma'(1)=1} (the label of \w[),]{c\sb{1}}
for a cyclic ordering.

In the case of a dihedral ordering, we further normalize by
requiring that \w{\sigma'(2)} appear as the label of a cyclotomic point in the upper
half-plane (unless it is at \w[,]{c\sb{n/2}=\pi}, in which case \w{\sigma'(3)} must be in the
upper half-plane).

We denote the subgroup of \w{\AG} fixing a given cyclic (or dihedral) ordering
\w{[\sigma]} by \w[.]{\Aut([\sigma])} Thus \w{\psi\in\Aut([\sigma])} if and only if
\w[.]{\sigma'=\sigma}
\end{definition}

\begin{subsection}{Fine cells}
\label{sfinecell}
Note that \w{\AG} acts freely on the dense open subspace \w[,]{\wCoG} but
\w{\Aut([\sigma])} takes the open cell \w{\wE\sb{\sigma}} to itself. Thus the
cardinality $N$ of \w{\Aut([\sigma])} is the number of fine cells into which we
must divide \w[.]{\wE\sb{\sigma}}

To specify such a fine cell, we think of \w{\Aut([\sigma])} as a subgroup of
the cyclic group \w{C\sb{n}} (or the dihedral group \w[,]{D\sb{n}} in the fully reduced
case), now acting in the standard way on the regular $n$-gon (or the cyclotomic points
on the circle). Each orbit under this action imposes a (different) partition of
\w{\wE\sb{\sigma}} into fine cells, defined as follows:

\begin{enumerate}
\renewcommand{\labelenumi}{(\alph{enumi})}
\item If \w{\Aut([\sigma])} is cyclic of order $N$, generated by
$$
  \psi\in\Aut([\sigma])~\subseteq~\Aut(\Gamma)~\subseteq~D\sb{2n}~\subseteq~
  \Aut(\Delta\sb{n})~=~S\sb{n}
$$
  (which is always true for the reduced \cspace), then $\psi$ acts on (the labelling of)
  the cyclotomic points \w{(c\sb{k})\sb{k=1}\sp{n}} as a rotation by an angle of
  \w[,]{\frac{2\pi}{N}} so the orbit of \w{c\sb{\ell}}  under this action is
  \w[,]{\alpha=(c\sb{i\sb{k}})\sb{k=1}\sp{N}} where \w[.]{i\sb{k}=\ell+(k-1)N\pmod{n}}

For each such orbit $\alpha$, we define one (open) fine cell
\w{\hF\sb{\sigma}(i\sb{k},i\sb{k+1})}  of \w{\hE\sb{\sigma}} for each element
\w{c\sb{i\sb{k}}} in the orbit: this consists of those configurations (angle sets)
\w{\vth\in\hE\sb{\sigma}} for which the angle
\w{|\theta\sb{i\sb{k}}-\theta\sb{i\sb{k}+1}|} is strictly smaller than
\w{|\theta\sb{i\sb{j}}-\theta\sb{i\sb{j}+1}|} for all \w[.]{1\leq j\neq k\leq N}

The same rule applies for the subdivision of \w{\wE\sb{\sigma}} in \w{\wCoG}
into  fine cells \w[.]{\wF\sb{\sigma}(i\sb{k},i\sb{k+1})}
\item If \w{\Aut([\sigma])} is dihedral of order \w{N=2M} (necessarily in the fully
  reduced case), we distinguish three cases:
\begin{enumerate}
\renewcommand{\labelenumii}{(\roman{enumii})}
\item If we have an orbit with three neighboring cyclotomic points, it necessarily includes
  \emph{all} the cyclotomic points \wh in other words, $\sigma$ is the identity
  permutation, and the open cell \w{\wE\sb{\sigma}} is divided into $N$ fine cells.
 These may be distinguished  by specifying for which \w{1\leq k\leq N} the angle
 \w{\theta\sb{k}-\theta\sb{k}+1} is smallest, and then further dividing this set into
 two subcells by the two possible orderings of \w{|\theta\sb{k+1}-\theta\sb{k+2}|} and
\w[.]{|\theta\sb{k-1}-\theta\sb{k}|} This will determine a canonical
  labelling of each configuration \w{\bx\in\wE\sb{\sigma}} in the fine cell
   \w[,]{\wF\sb{\sigma}(i\sb{k},i\sb{k-1})} starting at the
  \wwd{i\sb{k}}th vertex in the direction \w[;]{i\sb{k}\mapsto i\sb{k-1})}
   similarly \w{\wF\sb{\sigma}(i\sb{k},i\sb{k+1})} for the other.
\item If we have an orbit with exactly two neighboring cyclotomic points \wh so the whole
  orbit consists of such pairs \w{(i\sb{k},i\sb{k}+1)} \wb{k=1,\dotsc M} \wwh we may
  again name a fine cell \w[,]{\wF\sb{\sigma}(i\sb{k},i\sb{k}-1)} say, by specifying
  which \w{i\sb{k}} has the smallest angle difference between \w{\theta\sb{i\sb{k}}}
and its cyclotomic neighbor \w{\theta\sb{i\sb{k}\pm 1}} which is \emph{not} in the orbit.
\item If we have an orbit with no neighboring
  cyclotomic points, we may name a fine cell \w{\wF\sb{\sigma}(i\sb{k},i\sb{k}-1)}
  or \w{\wF\sb{\sigma}(i\sb{k},i\sb{k}+1)} by specifying which \w{i\sb{k}} has the
  smallest angle \w[,]{|\theta\sb{i\sb{k}}-\theta\sb{i\sb{k}\pm1}|} as in (i).
\end{enumerate}
\end{enumerate}

\begin{definition}\label{dmembrane}
For each open cell \w{\hE\sb{\sigma}} of \w{\hCG} corresponding to a cyclic ordering
\w{[\sigma]} as above, the \emph{membrane} separating two fine cells
\w{\hF\sb{\sigma}(i\sb{k},i\sb{k}+1)} and \w{\hF\sb{\sigma}(j\sb{m},j\sb{m}+1)}
is the subset
\w{\cM=\overline{\hF\sb{\sigma}(i\sb{k},i\sb{k}+1)}\cap
  \overline{\hF\sb{\sigma}(j\sb{m},j\sb{m}+1)}}
determined by the condition
\begin{equation}\label{eqredmemb}
|\theta\sb{i\sb{k}}-\theta\sb{i\sb{k}+1}|~=~
|\theta\sb{j\sb{m}}-\theta\sb{j\sb{m}+1}|~.
\end{equation}

Similarly, for each open cell \w{\wE\sb{\sigma}} of \w{\wCG} corresponding to a
dihedral ordering \w[,]{[\sigma]} the membrane
separating two fine cells \w{\wF\sb{\sigma}(i\sb{k},i\sb{k'})}
and \w{\wF\sb{\sigma}(j\sb{m},j\sb{m'})} is the subset
\w{\cM=\overline{\wF\sb{\sigma}(i\sb{k},i\sb{k'})}\cap
  \overline{\wF\sb{\sigma}(j\sb{m},j\sb{m'})}}
determined by one of two conditions:
\begin{enumerate}
\renewcommand{\labelenumi}{(\alph{enumi})}
\item If \w[,]{i\sb{k}=j\sb{m}} necessarily \w{k'\neq m'} \wwh say, \w{k'=k-1} and
\w{m'=m+2} \wwh and $\cM$ is then determined by the condition
\begin{equation}\label{eqfredmemba}
|\theta\sb{i\sb{k}}-\theta\sb{i\sb{k}-1}|~=~
|\theta\sb{j\sb{m}+1}-\theta\sb{j\sb{m}+2}|~.
\end{equation}
\item If \w[,]{i\sb{k}\neq j\sb{m}} $\cM$ is determined by the condition
\begin{equation}\label{eqfredmembb}
|\theta\sb{i\sb{k}}-\theta\sb{i\sb{k}+1}|~=~
|\theta\sb{j\sb{m}}-\theta\sb{j\sb{m}+1}|~.
\end{equation}
\end{enumerate}
\end{definition}
\end{subsection}

\begin{theorem}\label{ttriang}
If \w{\Gamma=\Gkc{n}} is a closed $n$-chain with length vector $\vel$,
its reduced \cspace\ \w{\hCG} has a regular \wwd{\AG}equivariant cell structure,
subordinate to the fine cell decomposition of \S \ref{sfinecell}, and thus to
the regular cell structure of \cite{PaninM}.
\end{theorem}

\begin{proof}
Since the membranes are
(codimension $1$) subspaces of the open cells \w{\hE=\hE\sb{\sigma}} in the dense open set
\w[,]{\hCoG\subseteq\hCG} we may use the identifications
\w{\hE\xra{\Phi}\hE'\xra{\Pc}\hE''} to try to describe the membrane $\cM$ using the chosen
affine coordinates for \w[.]{\hE''}

Since \w{\vt=(t\sb{1},\dotsc,t\sb{n})=\Phi(\vth)} for
\w[,]{t\sb{i}=\tan\frac{\theta\sb{i}}{2}} using the formula
$$
\tan(\alpha-\beta)~=~\frac{\tan(\alpha)-\tan(\beta)}{1+\tan(\alpha)\cdot\tan(\beta)}~,
$$
\noindent we see that \wref{eqredmemb} takes the form:
\begin{equation}\label{eqtanequal}
\frac{t\sb{i\sb{k}}-t\sb{i\sb{k}+1}}{1+t\sb{i\sb{k}}\cdot t\sb{i\sb{k}+1}}~=~
\frac{t\sb{j\sb{m}}-t\sb{j\sb{m}+1}}{1+t\sb{j\sb{m}}\cdot t\sb{j\sb{m}+1}}
\end{equation}
(assuming for the moment that \w{\theta\sb{i\sb{k}}>\theta\sb{i\sb{k}+1}} and
\w[).]{\theta\sb{j\sb{m}}>\theta\sb{j\sb{m}+1}}

For simplicity, let \w{i\sb{k}=i\sb{1}=1} (so \w[):]{t\sb{1}=0} writing
\w[,]{i=i\sb{k}+1} \w[,]{j=j\sb{m}} and \w[,]{k=j\sb{m}+1} \wref{eqtanequal} simplifies
to:
\begin{equation}\label{eqtanequalspec}
t\sb{i}\cdot t\sb{j}\cdot t\sb{k}~=~t\sb{k}-t\sb{j}-t\sb{i}~.
\end{equation}

If we further assume that \w{\{i,j,k\}\cap\{1,2,n\}=\emptyset}
(after identifying \w{\wE\sb{\sigma}} with \w{\wE\sb{\Id}} under
an appropriate relabelling of the arrows), then
\w[,]{\theta\sb{i}} \w[,]{\theta\sb{j}} and \w{\theta\sb{k}} are
taken under $\Pc$ to the corresponding coordinates
\w{(s\sb{i},s\sb{j},s\sb{k})} in the affine cell \w[,]{\wE''} as
in \wref[.]{eqaffinecell}

Substituting the values for \wref{eqmoebiusinv} into \wref{eqtanequalspec} yields:
\begin{equation}\label{eqtesmoeb}
\begin{split}
&\frac{t\sb{2}\sp{3}t\sb{n}\sp{3}}{((t\sb{2}-t\sb{n})s\sb{i}+t\sb{n})\cdot((t\sb{2}-t\sb{n})s\sb{j}+t\sb{n})\cdot
    ((t\sb{2}-t\sb{n})s\sb{k}+t\sb{n})}\\
&=~\frac{t\sb{2}t\sb{n}}{(t\sb{2}-t\sb{n})s\sb{k}+t\sb{n}}~-~
\frac{t\sb{2}t\sb{n}}{(t\sb{2}-t\sb{n})s\sb{j}+t\sb{n}}~-~
\frac{t\sb{2}t\sb{n}}{(t\sb{2}-t\sb{n})s\sb{i}+t\sb{n}}~,
\end{split}
\end{equation}
\noindent or, after cross-multiplying:
\begin{equation}\label{eqtesmoebius}
\begin{split}
t\sb{2}^{2}t\sb{n}^{2}~=&~
((t\sb{2}-t\sb{n})s\sb{i}+t\sb{n})\cdot((t\sb{2}-t\sb{n})s\sb{j}+t\sb{n})\\
    ~&-~((t\sb{2}-t\sb{n})s\sb{i}+t\sb{n})\cdot((t\sb{2}-t\sb{n})s\sb{k}+t\sb{n})\\
    ~&-~((t\sb{2}-t\sb{n})s\sb{j}+t\sb{n})\cdot((t\sb{2}-t\sb{n})s\sb{k}+t\sb{n})
\end{split}
\end{equation}

Note that if \w{t=\tan(\theta/2)} then (by the usual rational parametrization
of \w[):]{S\sp{1}}
\begin{equation}\label{eqcossin}
  \cos\theta~=~\frac{1-t\sp{2}}{1+t\sp{2}}\hsp\text{and}\hsp
  \sin\theta~=~\frac{2t}{1+t\sp{2}}~,
\end{equation}
\noindent so \wref{eqsumsin} becomes:
\begin{equation}\label{eqvecsums}
\sum\sb{i=1}\sp{n}\ \ell\sb{i}\cdot\frac{1-t\sb{i}\sp{2}}{1+t\sb{i}\sp{2}}~=0
\hsp\text{and}\hsp
\sum\sb{i=1}\sp{n}\ \ell\sb{i}\cdot\frac{2t\sb{i}}{1+t\sb{i}\sp{2}}~=~0~.
\end{equation}
\noindent Setting \w[,]{t\sb{1}=0} and substituting:
\begin{equation}\label{eqmoebiusinverses}
t\sb{i}~=~\frac{t\sb{2}t\sb{n}}{(t\sb{2}-t\sb{n})s\sb{i}+t\sb{n}}~.
\end{equation}
\noindent from \wref{eqmoebiusinv} into \wref{eqvecsums} for \w[,]{i=3,\dotsc,n-1}
we obtain two equations involving
\w[.]{\{t\sb{1}=0,t\sb{2},s\sb{3},\dotsc,s\sb{n-1},t\sb{n}\}}
We can solve these for \w{t\sb{2}}
and \w[,]{t\sb{n}} obtaining an algebraic equation for the membrane in the variables
\w{\vvs=(s\sb{3},\dotsc,s\sb{n-1})} from \wref[.]{eqtesmoeb} The fine cells bounded by
this membrane are therefore semi-algebraic sets inside the affine sets (thought of
as open subsets of the standard affine space \w[),]{\RR{n-3}} in which the equalities
\wref[,]{eqredmemb} \wref[,]{eqfredmemba} and \wref[,]{eqfredmembb} are replaced by
inequalities.

We may now use Hironaka's result on triangulating real semi-algebraic sets
(see \cite{HiroT}) to obtain the required triangulation, which can be made compatible
with that of the boundaries by \cite{LojaT}. We do so by induction on the dimension
of the naive cells: the main point to keep in mind is that once we choose a
triangulation for one of the fine cells, on one side of a membrane, we may reproduce it
for all the others using the free action of \w{\AG} on the interior (and the fact
that this action is compatible with that on the boundary of the fine cell,
which may not be free).
\end{proof}

\begin{corollary}\label{ctriang}
For \w{\Gamma=\Gkc{n}} as above, the fully reduced \cspace\ \w{\wCG}
also has a regular \wwd{\AG}equivariant cell structure, compatible with that
of \w{\hCG} in Theorem \ref{ttriang}.
\end{corollary}

\begin{proof}
The \wwd{C\sb{2}}action on \w{\hCG} by reflections in the $x$-axis generally
switches two distinct cells in the regular cell structure of Section \ref{cgppoly},
thus form a double cover of a single cell in \w[.]{\wCG} The only cells fixed by
the \wwd{C\sb{2}}action are vertices corresponding to linear configurations
(with all adjacent angles $0$ or $\pi$), and these can be dealt with as in
\cite[\S 4]{PaninM} (see also \cite{GPaninM}).
\end{proof}

%
%
\section{Symmetric configurations for planar polygons}
\label{cscpp}

In Section \ref{capp} we described a process for producing an
equivariant cell structure for the (fully) reduced \cspace\ of a
planar polygon, based on a refinement of the non-equivariant
regular cell structure described in Section \ref{cgppoly}. When
the \wwd{\AG}action on the fine cells is free, the orbit yields a
single cell in the associated symmetric \cspace\ \w[.]{\wSG}  As
noted above, this will happen in the interior of
the top dimensional cells. However, in the lower dimensional
cells, occurring in the boundary of those with free action, we may
have fixed points. From the proof of Theorem \ref{ttriang} we see
that in fact we must start from the lowest dimensional cells in
constructing our equivariant cell structure. Thus, we need to
understand the \emph{symmetric} configurations: those fixed under
a subgroup $H$ of \w[).]{\AG} In particular, these will be needed
in order to obtain a compatible decomposition of \w[.]{\wSG}

By Lemma \ref{pautograph}, for any $n$-polygon $\Gamma$ the automorphism group \w{\AG}
is either cyclic or dihedral, so the same is true for any subgroup \w[.]{H\leq\AG}
We must consider three cases\vsm:

\noindent\textbf{Case I.}\ Cyclic subgroup of \w{\AG} generated by a reflection\vsm:

Assume first that the subgroup $H$ is generated by a reflection
\w[,]{\rho\in\AG} in an axis of symmetry $\lambda$ for the labels
\w{(A\sb{1},\dotsc,A\sb{n})} of $\Gamma$. As noted in the proof of
Lemma \ref{pautograph}, this can have three forms:

\begin{enumerate}
\renewcommand{\labelenumi}{(\arabic{enumi})}
\item If \w{n=2k+1} is odd, $\rho$ is a reflection in a \emph{median}
  connecting \w{A\sb{k+1}} with the midpoint of \w[,]{A\sb{n}A\sb{1}} say, so
  \w{\rho(A\sb{i})=A\sb{n+1-i}} \wb[.]{i=1,\dotsc,k+1}
\item If \w{n=2k} is even, $\rho$ may be a reflection in
\begin{enumerate}\renewcommand{\labelenumii}{(\roman{enumii})}
\item  a \emph{diagonal}, connecting \w{A\sb{1}} with \w[,]{A\sb{k+1}} say, so
\w{\rho(A\sb{i})=A\sb{n+2-i}} (indices taken modulo $n$).
\item a \emph{midsegment}, connecting the midpoint of \w{A\sb{n}A\sb{1}} with that
  of \w[,]{A\sb{k}A\sb{k+1}} say, so \w[.]{\rho(A\sb{i})=A\sb{n+1-i}}
\end{enumerate}
\end{enumerate}

\begin{proposition}\label{lreflsym}
  Let \w{\Gamma=\Gkc{n}} be a planar $n$-polygon, with \w{H=C\sb{2}\leq\AG}
  generated by a reflection in the axis $\lambda$. We then have a map
  \w{\phi:\wCG\sp{H}\to\wC(\Gkl)} from the fixed-point set, where:
\begin{enumerate}
\renewcommand{\labelenumi}{(\arabic{enumi})}
\item If \w{n=2k+1} and $\lambda$ is a median to \w[,]{AB} \w{\Gkl} is either
  half of \w[,]{\Gamma'=\Gamma\setminus(AB)} and $\phi$ is one-to-one.
\item If \w{n=2k} and $\lambda$ is a diagonal, \w{\Gkl} is either half of $\Gamma$,
  and $\phi$ is one-to-one except over \w{\Gkc{k}} (the subspace of
  chains which close up), where the fiber is \w[.]{\RP{1}}
\item If \w{n=2k+2} is even and $\lambda$ is a midsegment from \w{AB} to \w[,]{CD}
  \w{\Gkl} is either half of \w[,]{\Gamma'=\Gamma\setminus(AB),(CD)} and
  $\phi$ is a double cover.
\end{enumerate}
\end{proposition}

\begin{proof}
In each case, by reflecting a given configuration for the open chain \w{\Gkl} about
  the geometric axis of symmetry $L$, we obtain a unique symmetric configuration
  for $\Gamma$.
\begin{enumerate}
\renewcommand{\labelenumi}{(\arabic{enumi})}
\item When $\lambda$ is a median, $L$ is the $y$-axis (the perpendicular
  bisector of \w[)]{AB} in the fully reduced case.  We may then rotate any configuration
  for \w{\Gkl} about $B$ until it touches $L$.
\item When $\lambda$ is a diagonal \w[,]{AD} say, for any configuration $\bx$ of the open
  chain \w{\Gkl} we let $L$ be the line from \w{\bx(A)} to \w[;]{\bx(D)}  however,
  when these two coincide we may choose $L$ at will from \w[.]{\RP{1}}
\item When $\lambda$ is a midsegment from \w{AB} to \w[,]{CD} $L$
is the perpendicular bisector of \w[;]{AB} let \w{L'}  and \w{L''}
be
  the two parallels to $L$ at distance \w[,]{\frac{1}{2}d(C,D)} and then rotate
any configuration for \w{\Gkl} about $B$ until it touches \w{L'} or \w[.]{L''}
\end{enumerate}

See Figure \ref{fmidref} for the third case.
\end{proof}
%
%
\begin{figure}[ht]
\centering{\includegraphics[width=0.2\textwidth]{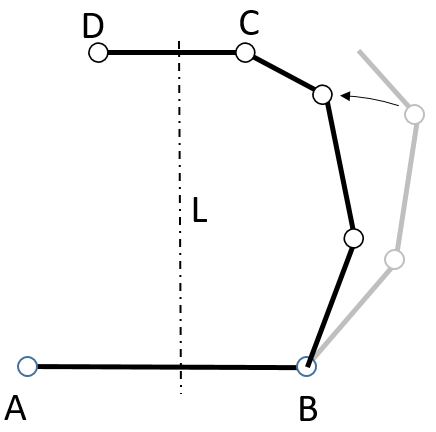}}
\caption{Midsegment reflection}
\label{fmidref}
\end{figure}

\begin{remark}\label{rscell}
In principle, we would like to relate the analysis of the symmetric configurations in
\w{\hCG} or \w{\wCG} with the fine cell structure introduced in Section \ref{capp}.
This would require a better understanding of the triangulation described in
Theorem \ref{ttriang}.  However, it is possible to use Proposition \ref{lreflsym}
to study the action of \w{\AG} on the open cell \w{E=E\sb{\Id}} of (strictly)
convex configurations, under a single reflection
\w{\rho\in\AG} in an axis of symmetry $\lambda$ for the labels
\w{(A\sb{1},\dotsc,A\sb{n})} of $\Gamma$.

The set $D$ of convex configurations \w{\bx\in E} which are fixed
under $\rho$ must also be symmetric in the geometric sense, with
axis of symmetry $L$ realizing $\lambda$. Such a configuration is
completely described by the half on one side of $L$, which is
simply a fully reduced convex configuration $\by$ for an open
chain of length \w[.]{k\approx\frac{n}{2}} As in \S
\ref{egopench}, we parameterize such a $\by$ by
\w[,]{\vth(\by)=(\theta\sb{1},\dotsc,\theta\sb{k})} (see \S
\ref{sparamcs}), with
$$
0~=~\theta\sb{1}<\theta\sb{2}<\dotsc<\theta\sb{k}~<~\pi~.
$$
\end{remark}

Switching to the parametrization \wref[,]{eqtphi} we have:

\begin{lemma}\label{lopenchk}
The set of strictly convex fully reduced configurations of a
closed chain \w{\Gamma=\Gkc{n}} symmetric with respect to a
reflection in $L$ as above are parameterized by
\begin{equation}\label{eqvtt}
\vt~=~(t\sb{1},\dotsc,t\sb{k})~=~(\tan(\theta\sb{1}/2),\dotsc,\tan(\theta\sb{k}/2))~.
\end{equation}
\noindent with \w[.]{0=t\sb{1}<t\sb{2}<\dotsc<t\sb{k}<\infty}
\end{lemma}

The precise value of $k$ will appear in the proof.

\begin{proof}
As in \wref[,]{eqvecsums} we may calculate the vector sum of the arrow diagram \w[,]{\vth}
using the length vector $\vel$ for $\Gamma$, by
\begin{equation}\label{eqvv}
\vv~:=~\left(\sum\sb{i=1}\sp{k}\ \ell\sb{i}\cdot\frac{1-t\sb{i}\sp{2}}{1+t\sb{i}\sp{2}},~
\sum\sb{i=1}\sp{k}\ \ell\sb{i}\cdot\frac{2t\sb{i}}{1+t\sb{i}\sp{2}}\right)~=~
(a,b)~=~\mu(\cos\theta,\sin\theta)~,
\end{equation}
\noindent where \w{\mu=\sqrt{a\sp{2}+b\sp{2}}} and \w[.]{\theta=\arctan(b/a)}

We distinguish three cases in calculating the slope $\tau$ of $L$ from $\vt$:

\begin{enumerate}
\renewcommand{\labelenumi}{(\arabic{enumi})}
\item If \w{n=2k+1} and $L$ is the perpendicular from \w{\by(A\sb{k})=\vv} to the
  midpoint of \w{\by(e\sb{n})} (because of the symmetry of $\by$), which has length
  \w{\ell\sb{n}} and forms an (unknown)
  angle of \w{\theta\sb{n}} with the positive $x$ axis, so $L$ forms an angle of
  \w[.]{\theta\sb{n}-\pi/2} Since $L$, $\vv$, and \w{\by(e\sb{n})} form a right
  triangle with hypotenuse $\mu$ and edge \w{\ell\sb{n}/2}
  facing the angle \w[,]{\alpha=\pi/2+\theta-\theta\sb{n}} with
  \w[.]{\sin\alpha=\frac{\ell\sb{n}/2}{\mu}} This allows us to recover $\alpha$, and
  thus \w{\theta\sb{n}} and find the slope of $L$, from which we can recover the
  remaining angles \w[.]{\theta\sb{k+1},\dotsc,\theta\sb{n-1}}
\item If \w{n=2k} and $L$ is a diagonal, its direction is the vector $\vv$ of
  \wref[,]{eqvv} from which we can calculate \w{\theta\sb{k+1},\dotsc,\theta\sb{n}}
  by reflecting \w{\theta\sb{1},\dotsc,\theta\sb{k}} in $L$.
\item If \w{n=2(k+1)} and $L$ is the perpendicular bisector of \w{\by(e\sb{k+1})}
  and \w[,]{\by(e\sb{n})} it forms an angle $\alpha$ with $\vv$, with
  \w[,]{\sin\alpha=\frac{|\ell\sb{n}/2-\ell\sb{k+1}/2|}{\mu}} from which we can
  calculate $\alpha$, and thus the remaining angles
  \w[.]{\theta\sb{k+1},\dotsc,\theta\sb{n}}
\end{enumerate}
\end{proof}

\noindent\textbf{Case II.}\ Cyclic subgroup of \w{\AG} generated by a rotation\vsm:

In the case where the cyclic subgroup \w{H\leq\AG} is generated by a rotation we have:

\begin{proposition}\label{lsymmetcellcy}
Let $\Gamma$ be a planar $n$-polygon and \w{H=C\sb{d}} a rotation subgroup of
\w[,]{\AG} which we identify with a subgroup of \w{D\sb{2d}}
for \w[,]{d|n} and let \w[.]{k:=n/d}  The fixed-point set \w{\wCG\sp{H}} is
isomorphic to a disjoint union, indexed by the discrete set
\w[,]{\wC(\Gen{d})\sp{H}} of copies of \w[,]{\hC(\Gkl)}  unless
\w{\wz\in\wC(\Gen{d})\sp{H}} is collinear, in which case we replace
\w{\hC(\Gkl)} by \w[.]{\wC(\Gkl)}
\end{proposition}

\begin{proof}
Note that $H$ is of index $2$ in \w[,]{D\sb{2d}} which we identify with
the automorphism group of the equilateral polygon \w{\Gen{d}} with $d$ edges
of length $L$ (unless \w[,]{d=2} in which case \w[).]{H=D\sb{2d}}

Any fully reduced configuration \w{\wx\in\wCG} is determined by
its restriction to an open subchain $\Delta$ of $\Gamma$ with \w{k=n/d} edges,
yielding a reduced configuration $\hy$ for $\Delta$, together with a fully reduced
configuration $\wz$ for \w[,]{\Gen{d}} with $L$ equal to the distance between
the start and end points of $\hy$. Here $\wz$ must be fixed under all
(geometric) rotations of \w[,]{\Gen{d}} and thus under the full automorphism
group \w[,]{\Aut(\Gen{d})=D\sb{2d}} since the angles at all vertices of the polygon
must be equal.

The fact that $\hy$ need not be \emph{fully} reduced means that we will generally
have two fully reduced configurations \w{\pm\wy} attached to each fully symmetric
configuration of \w[.]{\Gen{d}}

Thus for the icosagon \w[,]{\Gen{20}} we have four configurations fixed
under \w{C\sb{5}}
for a given (non-collinear) fully reduced open chain configuration $\wy$ of length $4$,
`attached'' to the regular pentagon and pentagram on either side, as shown
in Figure \ref{ficosagon}, where each of the two pairs (a)-(b) and (c)-(d)
correspond to the reduced configurations \w[.]{\pm\wy}

%
%
\begin{figure}[ht]
\centering{\includegraphics[width=0.8\textwidth]{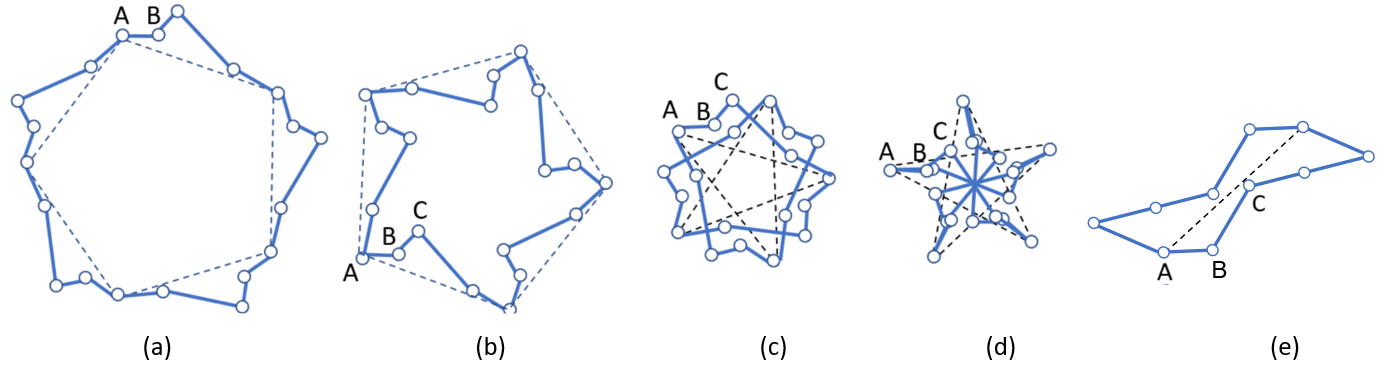}}
\caption{Symmetric configurations for the icosagon and decagon}
\label{ficosagon}
\end{figure}

On the other hand, the
(fully reduced) \wwd{C\sb{2}}configuration for the decagon shown in
  Figure \ref{ficosagon}(e) is the same for \w[,]{\pm\wy} so it requires only the fully
  reduced open loop $\wy$.

Note that even when \w{L=0} (that is, $\hy$ is actually a closed loop),
we still have distinct configurations $\hz$ corresponding to different
reduced symmetric configurations $\hz$: e.g., when \w[,]{d=5}, the convex pentagon and
the pentagram of Figure \ref{fpentagon} correspond to two different cyclic arrangements
of the petals of the bouquet of five such loops. It might be useful to think of the
common endpoints of all the closed loops as an infinitesimal reduced symmetric
configuration, in order to keep track of the orientations of the various loops.
This is illustrated by the five closed-loop configurations shown in Figure
\ref{ficosagoncl}, corresponding respectively to those of Figure \ref{ficosagon},
with the inner dashed symmetric configuration reduced to a point.

%
%
\begin{figure}[ht]
\centering{\includegraphics[width=0.8\textwidth]{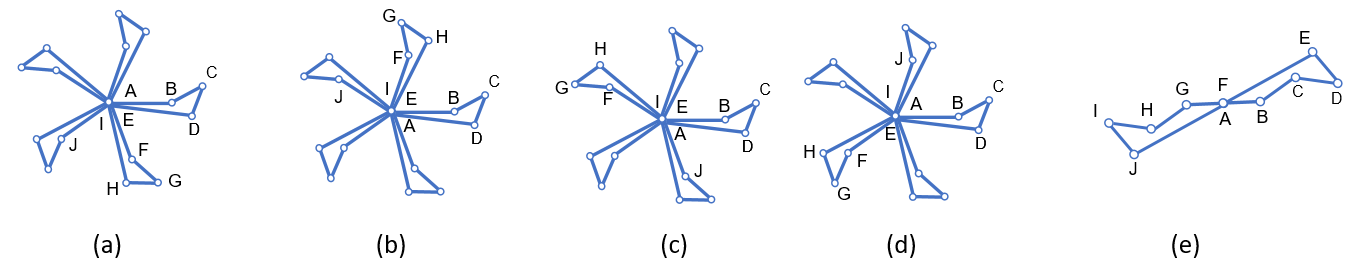}}
\caption{Closed loop configurations for the icosagon and decagon}
\label{ficosagoncl}
\end{figure}
\end{proof}

\begin{example}\label{egvertices}
  When \w[,]{d=6} the three possible configurations of the equilateral hexagon
  \w{\Gen{6}} which are invariant under the full automorphism group
  \w{D\sb{12}=\Aut(\Gen{6})} are the first three in the top row of
  Figure \ref{fhexvert} (see Remark \ref{rvertices}).
\end{example}

\begin{corollary}\label{ccycsymmcellp}
For $\Gamma$, \w{H=C\sb{d}} and \w{k=n/d}  as above, the interior of \w{\wCG\sp{H}}
has the same local parametrization as \w{\wC(\Gkl)} (see \S \ref{egopench}).
\end{corollary}

The case where the cyclic subgroup \w{H\leq\AG} is generated by a rotation is in fact
the only one relevant to the \emph{reduced} \cspace\ \w[,]{\hCG} where we have
the following somewhat simpler result:

\begin{proposition}\label{lsymmetcellcyred}
If $\Gamma$ is a planar $n$-polygon and \w{H=C\sb{d}} is a rotation subgroup of
\w[,]{\AG} the fixed-point set \w{\hCG\sp{H}} is isomorphic to a disjoint union,
indexed by the discrete set \w[,]{\wC(\Gen{d})\sp{H}} of copies of \w[\vsm.]{\hC(\Gkl)}
\end{proposition}

\noindent\textbf{Case III.}\ Dihedral subgroups\vsm:

Let $\Gamma$ be a planar $n$-polygon, \w{H\cong D\sb{2d}} a
dihedral subgroup of \w[,]{\AG\subseteq D\sb{2n}} with \w[.]{d|n}
Choose two generators $\rho$ and $\sigma$ for $H$, which we may
identify with reflections in axes $k$ and $m$, respectively, in a
regular $n$-gon \w[.]{\Gn{n}} Here we identify $H$ with a subgroup
of \w[,]{D\sb{2n}} even though $\Gamma$ need not be equilateral,
in order to have a consistent description of its automorphisms
(acting on a fully reduced configuration by relabelling).

If $n$ is odd, each axis is necessarily a median (connecting a vertex of
$\Gamma$ to the midpoint of the opposite edge). If $n$ is even, the axis could be
a midpoint interval (connecting the midpoints of two opposite edges) or a diagonal
connecting two opposite vertices.

The generator $\rho$  has a ``basic subchain'' $\Delta$ of
$\Gamma$ on which it acts by reflection in $k$ (under
relabelling): this is depicted in the blue segment \w{AB\dotsc
B'A'} in either of the two diagrams of Figure \ref{fscdr}.

When $k$ ends in a vertex (e.g., $C$ on the right in Figure \ref{fscdr}),
we have a ``fundamental subchain'' \w{\Delta'} (\w{ABC} in our example), reflected
under $\rho$ to \w{\Delta''} (i.e., \w[),]{A'B'C} with $\Delta$ the union of
\w{\Delta'} and \w[.]{\Delta''}

When $k$ ends in the midpoint of $\Delta$ of an edge \w{CC'} of $\Gamma$
(e.g., \w{N\sb{0}} on the left in Figure \ref{fscdr}), the ``fundamental subchain''
\w{\Delta'} ends in $C$ (so \w{\Delta'=ABC} in our example), and $\Delta$ is the union of
\w[,]{\Delta'} its reflection \w[,]{\Delta''} and the middle segment \w[.]{CC'}

Similarly, the generator $\sigma$  has a ``basic subchain'' $\Theta$ with
``fundamental subchain'' \w{\Theta'} (given by \w{AZYX} in both diagrams of
Figure \ref{fscdr}).

\begin{theorem}\label{lsymmetcelldy}
Let $\Gamma$ be a planar $n$-polygon and \w{H\cong D\sb{2d}} is a dihedral subgroup of
\w[,]{\AG} generated by reflections in axes $k$ and $m$ in a regular $n$-gon
\w[.]{\Gn{n}}  The fixed-point set \w{\wC=\wCG\sp{H}} splits as a disjoint
union indexed by \w[.]{\wC(\Gen{d})\sp{H}}
For each \w[,]{\wz\in\wC(\Gen{d})\sp{H}} the corresponding component \w{\wC\sb{\wz}} of
$\wC$ fibers over an interval \w[.]{[0,L\sb{0}]}
The fiber over a value $L$ (the length of all edges of $\wz$) further fibers over
a closed interval \w{I=I\sb{L}} in the \w{\RR{}P\sp{1}\cong S\sp{1}} (the space of lines
in the plane through the barycenter \w[).]{\wz(O)} Finally, given a line
\w{\wx(k)} in $I$, let \w{\wx(m)} denote its rotation by an angle of \w{\pi/d} about
\w[;]{\wz(O)} the fiber of $\wC$ over \w{\wx(k)} is then isomorphic to
\w[.]{Y\sb{\wx(k)}\times Y'\sb{\wx(m)}}
\end{theorem}

\begin{corollary}\label{cdihsymmcellp}
For $\Gamma$, \w[,]{H=D\sb{2d}} and the two open chains \w{\Delta'} and \w{\Theta'}
as above, the interior of \w{\wCG\sp{H}} has a local parametrization given by
\w[.]{(0,L\sb{0})\times I\times\hC(\Delta')\times\hC(\Theta')}
\end{corollary}

\begin{proof}
Given any \w[,]{L>0} let \w{\Gen{d}} be an equilateral $d$-gon with edge length $L$
and let \w{\wz\in\wC(\Gen{d})\sp{H}} be a fully reduced fully symmetric configuration for
\w[.]{\Gen{d}} Let \w{A'} and \w{A''} be the first two vertices of the $d$-gon (so
\w{\wz(A')} is at the origin, \w{\wz(A'')} is in the positive direction of the $x$ axis,
and \w[).]{d(\wz(A'),\wz(A'')=L}

In order to determine a fully reduced symmetric configuration
$\wx$ for $\Gamma$ (invariant under relabelling in the given
subgroup $H$), we must choose a line \w{\wx(k)} through the
barycenter \w{\wz(O)} of $\wz$, which will serve as the axis of
the \emph{geometric} reflection realizing the action of $\rho$ by
reflecting the \emph{labels} in the ``combinatorial axis'' $k$.

We let \w{\wx(m)} be the line through \w{\wz(O)} forming an angle of
\w{\pi/d} with \w{\wx(k)} (realizing geometrically the reflection $\sigma$ in $m$),
and let \w{\wx(A)} be the reflection of the origin (which is \w[)]{\wz(A')}
in the geometric axis \w[.]{\wx(k)}

Note that \w{\wz(O)} is the center of the circle $\gamma$ circumscribing the regular
$d$-gon $\wz$ and \w{\wx(k)} bisects \w[,]{\angle\wz(A')\wz(O)\wz(A)} so
\w{\wx(m)} bisects \w[.]{\angle\wz(A'')\wz(O)\wz(A)} Since
\w[,]{d(\wz(O),\wx(A))=d(\wz(O),\wz(A'))}
\w{\wx(A)} is also on $\gamma$, so \w{\wx(A)} is also the reflection of
\w{\wz(A'')} in \w[.]{\wx(m)}

%
%
\begin{figure}[ht]
\centering{\includegraphics[width=0.7\textwidth]{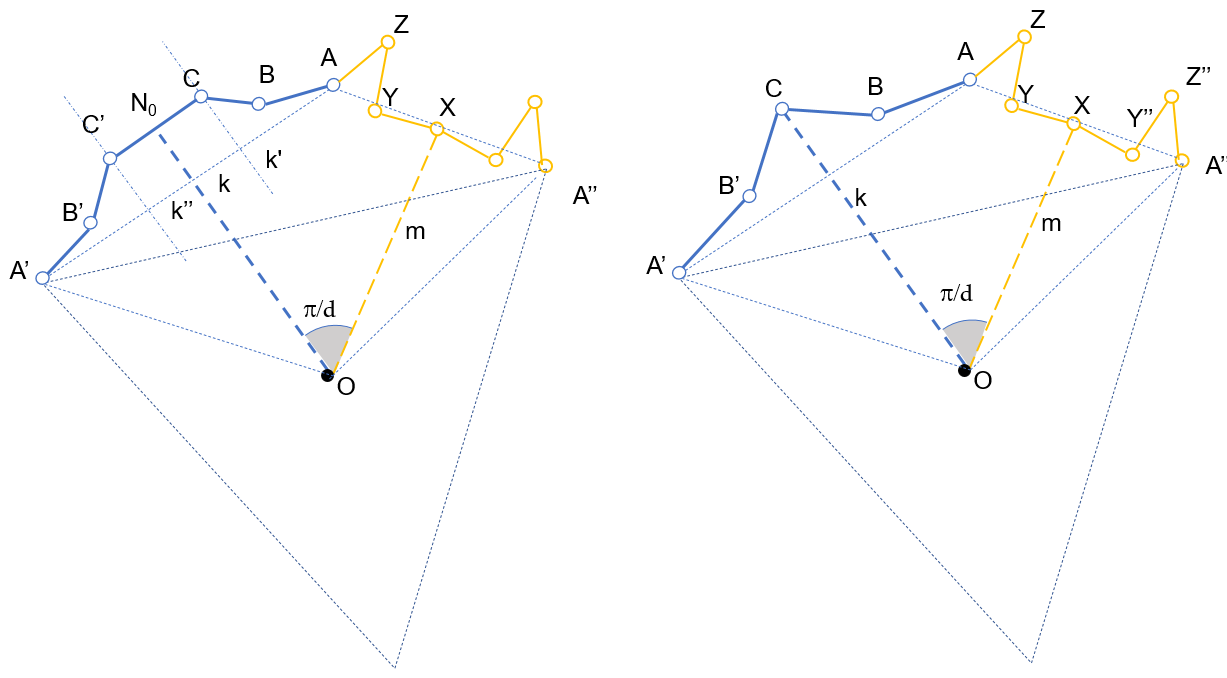}}
\caption{Symmetric configurations under double reflections}
\label{fscdr}
\end{figure}

In this situation, a pair of reduced configurations \w{(\hu,\hv)}
for the open chains \w{\Delta'} and \w[,]{\Theta'} respectively, generally will determine
(up to) four reduced configurations $\hy$ for the chain \w{\Delta\cup\Theta}
(the blue and yellow in Figure \ref{fscdr}), with endpoints \w{\hy(A')} and
\w{\hy(A'')} (such that \w[).]{d(\hy(A'),\hy(A'')=L}

To see how, we must distinguish two basic cases:

\begin{enumerate}
\renewcommand{\labelenumi}{(\alph{enumi})}
\item If the original axis $k$ (for the reflection $\rho$) ends in a vertex $D$,
  as on the right-hand side of in Figure \ref{fscdr}, the fundamental open subchain
  \w{\Delta'} of $\Gamma$ is \w[,]{(A,\dotsc,C)} say. Let
  \w{\lambda(\hu)=d(\hu(A),\hu(C))} be the length of a reduced configuration $\hu$ for
  \w{\Delta'} (a smooth function on the torus \w[).]{\hC(\Delta')} The circle
  \w{\gamma\sb{\hu}} of radius \w{\hu(\lambda)} about \w{\wx(A)} generally intersects
  \w{\wx(k)} in two points \w{\wx'(C)} and \w{\wx''(C)} (which coincide if
  \w[).]{d(\wx(A),\wx(k))=\lambda(\hu)} If this happens, we say that the reduced
  configuration $\hu$ of \w{\Delta'} is \emph{allowable} with respect to
  \w[.]{(\wz,\wx(k))}

  Allowable configurations $\hv$ of \w{\Theta'} are defined similarly if the axis $m$ for
  $\sigma$ ends in a vertex of $\Gamma$, when the circle \w{\gamma\sb{\hv}} of
  radius \w{\mu(\hv)} about \w{\wx(A)} intersects \w[.]{\wx(m)}
\item If $k$ ends in the midpoint \w{N\sb{0}} of an edge of $\Gamma$ (of length
  \w[,]{\ell\sb{i}} say), as on the left hand side of in Figure \ref{fscdr}, let
\w{\wx(k')} be the line parallel to \w{\wx(k)} at a distance of \w{\ell\sb{i}/2}
on the same side as \w[.]{\wx(A)} In this case, a reduced configuration
  $\hu$ of \w{\Delta'} of distance \w{\lambda(\hu)=d(\hu(A),\hu(D))} will be
  \emph{allowable} with respect to \w{(\wz,\wx(k))} if the circle \w{\gamma\sb{\hu}}
  intersects \w[.]{\wx(k')}  Similarly for \w{\Theta'} if the axis $m$
  ends in a midpoint.
\end{enumerate}

A pair of allowable reduced configurations \w{(\hu,\hv)}
for the open chains \w{\Delta'} and \w[,]{\Theta'} determines a reduced configuration
$\hy$ for the chain \w[,]{\Delta\cup\Theta} by letting \w{\hy(C)} be one of the two
intersections of the circle \w{\gamma\sb{\hu}} with \w[,]{\wx(k)} and similarly for
the endpoint of \w[.]{\Theta'}

When \w[,]{L=0} we think of $\wz$ as the unique infinitesimal
  fully symmetric configuration for \w[.]{\Gen{d}}  We may then take \w{\wx(k)} to
  be the $x$-axis, say. This determines \w[,]{\wx(m)}  and of course, \w{\wx(A)}
  will remain at the origin.

Each reduced configuration $\hy$ for the chain \w{\Delta\cup\Theta} yields a unique fully
reduced configuration $\wx$ in the fixed-point set \w[,]{\wCG\sp{H}} by \emph{rotating}
$\hy$ about $\wz$, since by comparing angles, we see that the continuation of
\w{\wx(k)} beyond \w{\wz(O)} is a rotation of
\w[,]{\wx(m)} and conversely. See Figure \ref{ffscdr}.

%
%
\begin{figure}[ht]
\centering{\includegraphics[width=0.45\textwidth]{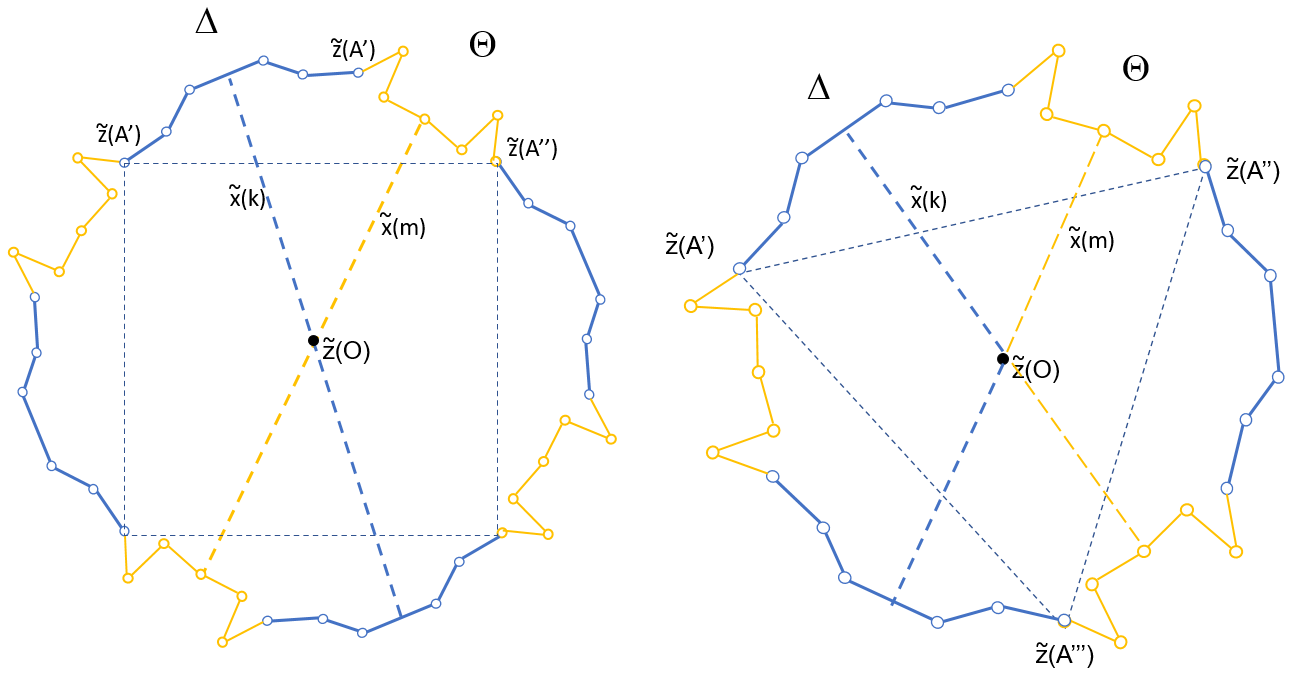}}
\caption{Fully symmetric configurations under double reflections}
\label{ffscdr}
\end{figure}

Let \w{\Gen{d}} be an equilateral $d$-gon with edge length $L$
and let \w{\wz\in\wC(\Gen{d})\sp{H}} be a fully reduced fully symmetric configuration
for \w[,]{\Gen{d}} as above.

Each line \w{\wx(k)} through \w{\wz(O)} determines a subspace \w{Y\sb{\wx(k)}} of
the pointed \cspace\ \w{\Cs(\Delta')} (a torus), consisting of all configurations
$\bu$ for \w{\Delta'} (starting at the origin) whose end-point \w{\bu(C)} lies on
$$
\begin{cases} \wx(k) & \text{when \w{\Delta'} has an even number
      of edges}\\
\wx(k'')&\text{when \w{\Delta'} has an odd number of edges}
\end{cases}
$$
\noindent See the right and left diagrams in Figure \ref{fscdr}, respectively.

We let \w{\wx(m)} be the line through \w{\wz(O)} forming an
angle of \w{\pi/d} with \w[,]{\wx(k)} and the subspace \w{Y'\sb{\wx(m)}} of
\w{\Cs(\Theta')} is defined analogously, with \w{\wx(m)} replacing \w[.]{\wx(k)}

To identify these subspaces, let \w{\mu=\mu(\wx(k))} denote the distance of the
origin \w{\wz(A')} from \w[,]{\wx(k)} and let
\w[,]{\mu\sb{\vel}:=\ell\sb{1}+\dotsc+\ell\sb{j}} where
\w{\vel=(\ell\sb{1},\dotsc,\ell\sb{j})} be the length vector of \w[:]{\Delta'}

\begin{enumerate}\renewcommand{\labelenumi}{(\roman{enumi})}
\item If \w[,]{\mu>\mu\sb{\vel}} clearly \w[.]{Y\sb{\wx(k)}=\emptyset}
\item \w[,]{\mu=\mu\sb{\vel}} then \w{Y\sb{\wx(k)}} consists of a single point:
the fully stretched configuration.
\item If \w[,]{0<\mu<\mu\sb{\vel}} we see that \w{Y\sb{\wx(k)}} is isomorphic to the
disjoint union of two copies of the reduced \cspace\ \w[,]{\hC(\widehat{\Delta'})}
where \w{\widehat{\Delta'}} is the \emph{closed} chain having length vector
\w[:]{\widehat{\vel}=(\ell\sb{1},\dotsc,\ell\sb{j},\mu)}
this is because for each reduced configuration \w[,]{\hw\in\hC(\widehat{\Delta'})}
we obtain two reduced configurations
\w{\hu\sb{1},\hu\sb{2}\in\hC(\widehat{\Delta'})} by rotating $\hw$ about the origin
so that the endpoint \w{\hw(C)} of \w{\Delta'} lies at one of the two intersections
of the circle of radius $\mu$ about the origin with \w[.]{\wx(k)}
\item If \w{\mu=0} \wwh that is, \w{\wx(k)} passes through the origin \wh
\w{Y\sb{\wx(k)}} decomposes into two complementary subspaces:
\begin{enumerate}\renewcommand{\labelenumii}{\arabic{enumii}.~}
\item \w[,]{Y\sb{0}} consisting of those configurations \w{\hu\in\Cs(\Delta')} for which
  \w{\bu(C)} is at the origin. This  may be identified with pointed configurations for
  the closed chain \w{\Gkc{j}} with length vector $\vel$, so
  \w[,]{Y\sb{0}\cong\wC(\Gkc{j})\times S\sp{1}} where the parameter
  \w{\phi\in S\sp{1}} determines the rotation of the reduced configuration
  \w{\hu\in\wC(\Gkc{j})} about the origin  (see \S \ref{sparamcs}).
\item \w[,]{Y\sb{1}} consisting of pointed configurations \w{\hu\in\Cs(\Delta')}
  not ending at the origin. These are again determined by rotating any reduced
  configuration $\hu$ for \w{\Delta'} about the origin, till its endpoint lies on one
  of the two intersections of the circle of radius \w{\lambda(\hu)} about the origin
  with the line \w[.]{\wx(k)}
\end{enumerate}
Note that \w{\Cs(\Delta')} is canonically isomorphic to \w{\wC(\Delta')\times S\sp{1}}
(see \S \ref{srestr}), which explains how both \w{Y\sb{0}} and \w{Y\sb{1}} embed
in \w[.]{\Cs(\Delta')}
\end{enumerate}

 We see that given \w{(\wz,L)} as above, the pair of configurations \w{(\hu,\hv)} for
 \w{\Delta'} and \w[,]{\Theta'} respectively, is allowable if and only if
\w{\hu\in Y\sb{\wx(k)}} and \w[.]{\hv\in Y'\sb{\wx(m)}}
Note that the maximal value of $L$ for which such allowable pairs exist is
\w[,]{L=L\sb{0}:=2\mu\sb{\vel}+2\mu'\sb{\vel'}+\nu} where $\nu$ is the sum of the lengths
of the middle edges of $\Delta$ and $\Theta$ (if these have an odd number of edges,
as in the right picture in Figure \ref{ffscdr}). In this case, there is a unique allowable
pair, yielding a fully stretched configuration for \w[.]{\Delta\cup\Theta}
\end{proof}

%
%
\section{Triangulating a cell for the hexagon}
\label{cfch}

As noted in Remark \ref{rscell}, a full \wwd{\AG}equivariant cell structure for
the fully reduced \cspace\ \w{\wCG} of a polygon $\Gamma$ requires a refinement
of the regular cell structure of Section \ref{capp}. We illustrate some of the issues
involved by considering a single regular cell \w{E=E\sb{\Id}} of strictly convex
configurations for the equilateral hexagon \w[.]{\Gamma=\Gen{6}} Note that $E$
itself is a bipyramid with six triangular facets, as in Figure \ref{fbipyr}
(in which the outer edges are to be identified pairwise as indicated).

%
%
\begin{figure}[ht]
\centering{\includegraphics[width=0.6\textwidth]{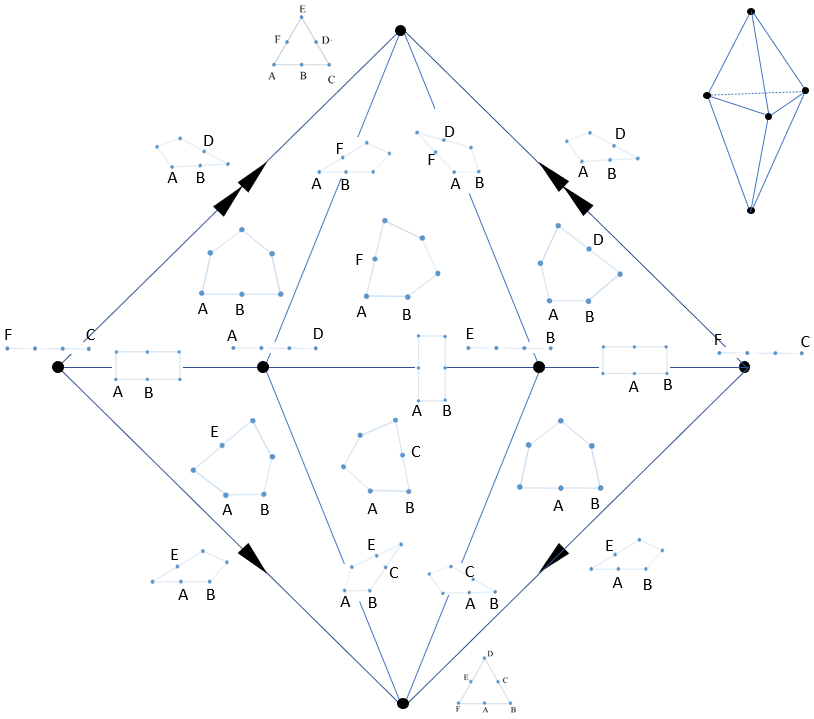}}
\caption{Bipyramid for the equilateral hexagon}
\label{fbipyr}
\end{figure}

\begin{remark}\label{rvertices}
The vertices of \w{\wCG} are determined by a combination of symmetries and
straightenings or foldings, which suffice to determine a rigid
configuration. The full list of all vertices for the equilateral hexagon in the regular
cell structure of Section \ref{capp} are of eleven types, depicted in
Figure \ref{fhexvert} (although, as we see in Figure \ref{fbipyr}, the same type may
appear with different labellings).

%
%
\begin{figure}[ht]
\centering{\includegraphics[width=0.9\textwidth]{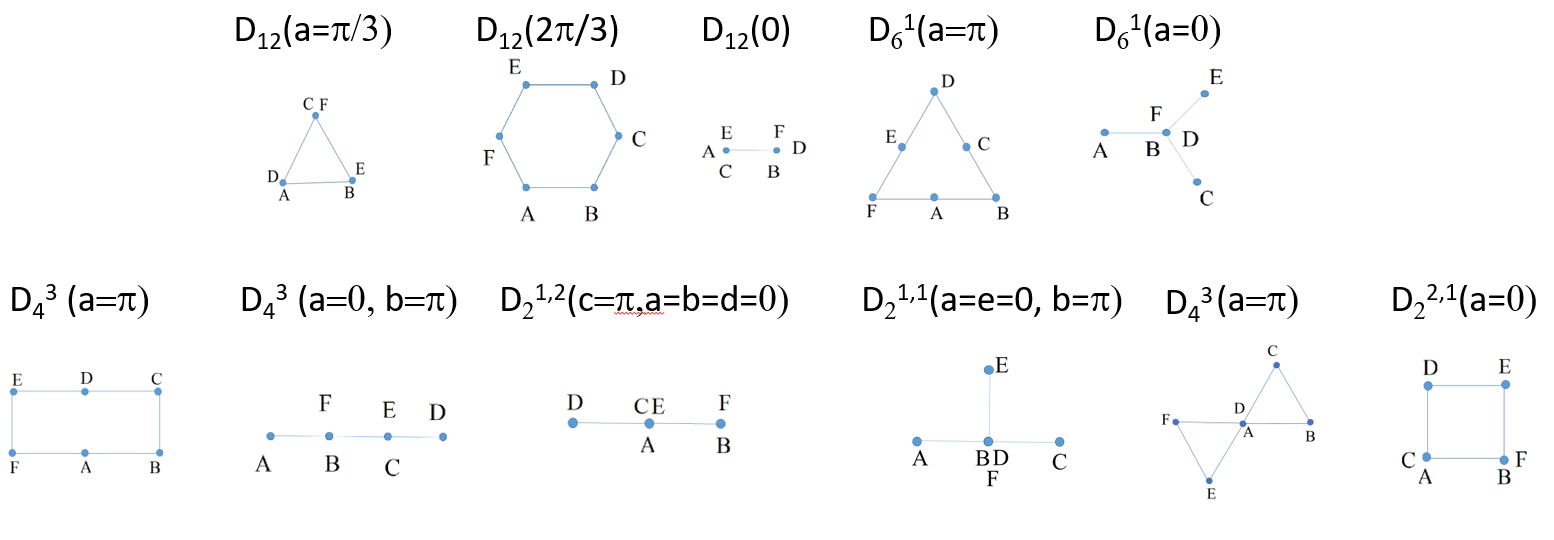}}
\caption{The vertex types of \w{\wCG} for the equilateral hexagon}
\label{fhexvert}
\end{figure}
\end{remark}

\begin{subsection}{The subdivided bipyramid}\label{ssubbipyr}
The bipyramid \w{E=E\sb{\Id}} of (strictly) convex configurations for
the equilateral hexagon \w[,]{\Gamma=\Gen{6}} should be subdivided into twelve
tetrahedra, which are permuted among themselves by action of \w{\AG} on the labels.
In accordance with the principles of \S \ref{sfinecell}, each tetrahedron
is determined by specifying the largest of the six angles of $\Gamma$,
and then choosing which of the two angles adjacent to it should be smaller.

This in Figure \ref{fhextetr} (on the left) we have required that the angle
\w{\theta\sb{B}} (labelled by $B$) should be the greatest, and that
\w[.]{\theta\sb{A}<\theta\sb{C}} One can then determine the induced inequalities
\w[,]{\theta\sb{A}<\theta\sb{D}<\theta\sb{F}} and
\w[,]{\theta\sb{C}<\theta\sb{E},\theta\sb{F}} as indicated in the lower left corner of
Figure \ref{fhextetr}.

%
%
\begin{figure}[ht]
\centering{\includegraphics[width=0.8\textwidth]{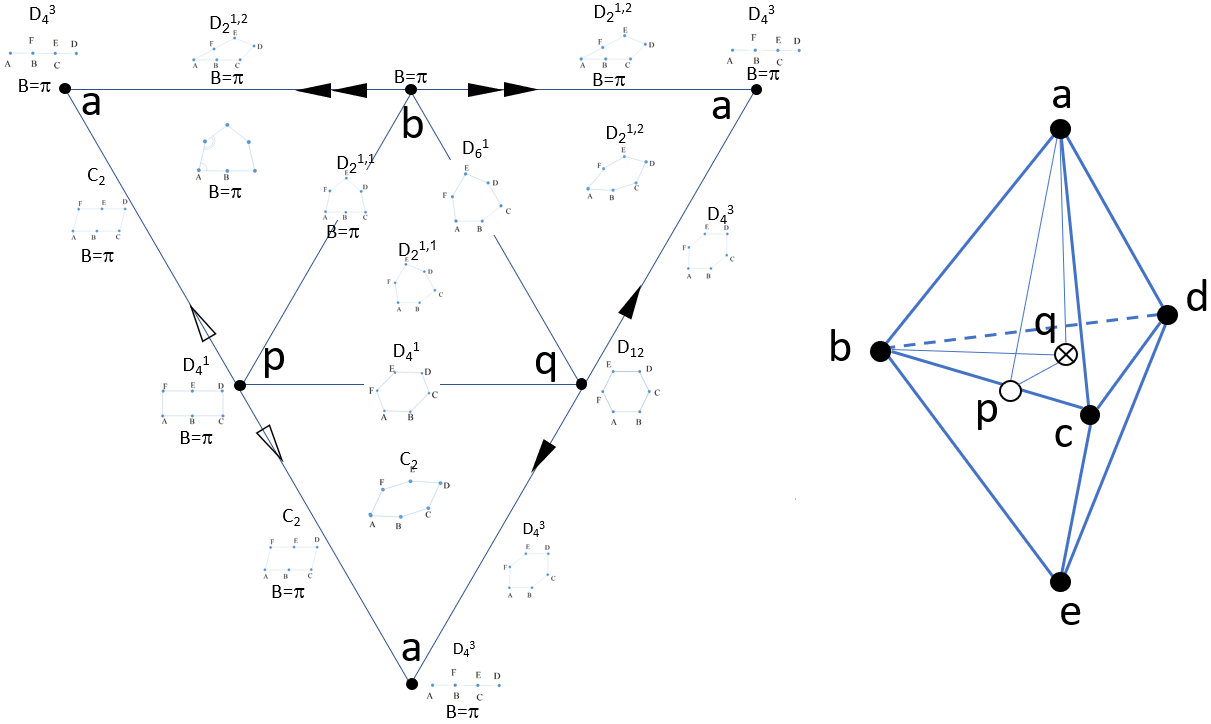}}
\caption{Fine $3$-cell for the equilateral hexagon \w{\Gamma=\Gen{6}} (one of twelve)}
\label{fhextetr}
\end{figure}

The boundary of the tetrahedron \w{abpq} consist of four triangular facets:

\begin{enumerate}
\renewcommand{\labelenumi}{(\alph{enumi})}
\item The boundary triangle \w{\triangle abp} is determined by the requirement
  that \w{\theta\sb{B}=\pi} (a straightening, which we abbreviate to \w{B=\pi} in the
  figure), so it is an open cell of the fully-reduced \cspace\ of the pentagon with length
  vector \w[.]{\vel=(2,1,1,1,1)} In turn its boundary consists of:
\begin{enumerate}
\renewcommand{\labelenumii}{\roman{enumii}.}
\item The edge \w[,]{ab} corresponding to the further straightening \w[,]{F=\pi}
  yields a deltoid of sides \w{\{2,1\}} and symmetry group \w{C\sb{2}}
  (corresponding to the subgroup \w{D\sb{2}\sp{1,2}} of \w{\AG=D\sb{12}} in
  Figure \ref{fsubgplat}, generated by the reflection in the diagonal \w{BE} of the
  hexagon).
\item The edge \w[,]{ap} corresponding to the straightening \w{E=\pi} yields
  a parallelogram of sides \w{\{2,1\}} and symmetry group \w{C\sb{2}} (corresponding
  to the subgroup \w{C\sb{2}<D\sb{12}} generated by the rotation by \w[).]{180\sp{o}}
\item The edge \w[,]{bp} corresponding to the symmetric version of the
  \wwd{(2,1,1,1,1)}pentagon, with \wwd{C\sb{1}}symmetry corresponding to
  \w[.]{D\sb{2}\sp{1,1}<D\sb{12}}
\end{enumerate}
\item The central triangle \w{\triangle bpq} in Figure \ref{fhextetr} is determined by
  requiring invariance under the subgroup \w{D\sb{2}\sp{1,1}} mentioned above.
\item The upper right triangle \w{\triangle abq} is invariant under the subgroup
\w{D\sb{2}\sp{1,2}} (generated by reflection in the diagonal \w[),]{AD} so the
common edge with the central triangle has symmetry group \w{D\sb{6}\sp{1}}
(generated by the two reflections and thus including the rotation by \w[).]{120\sp{o}}
\item The bottom triangle \w{\triangle apq} is invariant under the subgroup \w{C\sb{2}}
  (generated by the  \w{180\sp{o}} rotation. The edge \w{pq} consists of
  configurations invariant under the subgroup \w{D\sb{6}\sp{1}} of \w{D\sb{12}}
  generated by the reflections in \w{BF} and the median connecting the midpoints
  of \w{AF} and \w[.]{CD}
\end{enumerate}
\end{subsection}

\begin{remark}\label{rsubdivision}
  As noted above, the tetrahedron on the left of Figure \ref{fhextetr} appears as one of
  twelve subcells in the bipyramid of Figure \ref{fbipyr}, obtained by a barycentric
  subdivision as on the right in Figure \ref{fhextetr}: specifically, the upper left
  facet labelled \w{B=\pi} in Figure \ref{fhextetr} is one half of the upper left facet
  of Figure \ref{fbipyr}, ending at the center of the lower edge of the latter (the
  vertex corresponding to the rectangle marked \w{B=\pi} and \w{D\sb{4}\sp{1}} in the
  former).
  The vertex marked \w{D\sb{12}} in the tetrahedron is the barycenter $q$
  of the bipyramid in Figure \ref{fhextetr}, corresponding to the regular hexagon
  configuration. Observe that all other facets of the
  tetrahedron are symmetric \wh that is, fixed under an appropriate subgroup of
  \w[,]{\AG} as indicated in Figure \ref{fhextetr} \wh and are thus internal membranes
  of the bipyramid, in the language of \S \ref{dmembrane}.
\end{remark}

%
%
\section{The equilateral pentagon}
\label{ceqpent}

We now analyze in detail the case of the equilateral pentagon \w[.]{\Gen{5}}
Recall that \cite[\S 2.4]{HavD} identifies the reduced \cspace\ of \w{\Gen{5}}
(that is, \w{\hC(\Gen{5})} modulo orientation-preserving isometries)
as a genus $4$ oriented surface (see also \cite{KMillM}),
while \cite{KamiTE} shows that the fully reduced \cspace \w{\wCGn{5}} of
\S \ref{srestr} is the connected sum of five projective planes.
Note that Remark \ref{rfullyred} applies in this case.

\begin{subsection}{Cells for the pentagon}\label{scpent}
An analysis of the possible arrow diagrams for an equilateral pentagon shows
that there are only four dihedral types: the first, third, fourth, and fifth  in
Figure \ref{fpentagon}. Note that the first and second have the same dihedral type,
but different cyclic types (with reversed orientations), as we see from the
corresponding configurations (which are equivalent in the fully reduced
\cspace \w[,]{\wCG} but not in \w[).]{\hCG} The fourth and sixth also have the
same dihedral types, but distinct cyclic types.

In order to get a better grasp of the fine cell structure, it is convenient to use
here a slightly different labelling system, corresponding to open intervals of
allowable values for each of the five angles between consecutive edges
(as in \S \ref{sopench}).
We indicate the range \w{0<\theta<\pi} by $-$, and \w{\pi<\theta<2\pi} by $+$ (using
the convention of \S \ref{spcsp}).
Each sequence of the form \w[,]{(++---)} for example, defines a unique cell, except
for \w[,]{(-----)} which corresponds to two distinct cells as indicated in
Figure \ref{fpentagon} (where all five types are shown). The cell marked \w{(-----)}
has smallest angle  \w[,]{\geq 2\arcsin(0.25)\approx 0.5053} while the cell
marked \w{(-----)'} has largest angle \w{\leq \pi/3} (with dual conditions for
\w{(+++++)} and \w[).]{(+++++)'}

%
%
\begin{figure}[ht]
\centering{\includegraphics[width=1.0\textwidth]{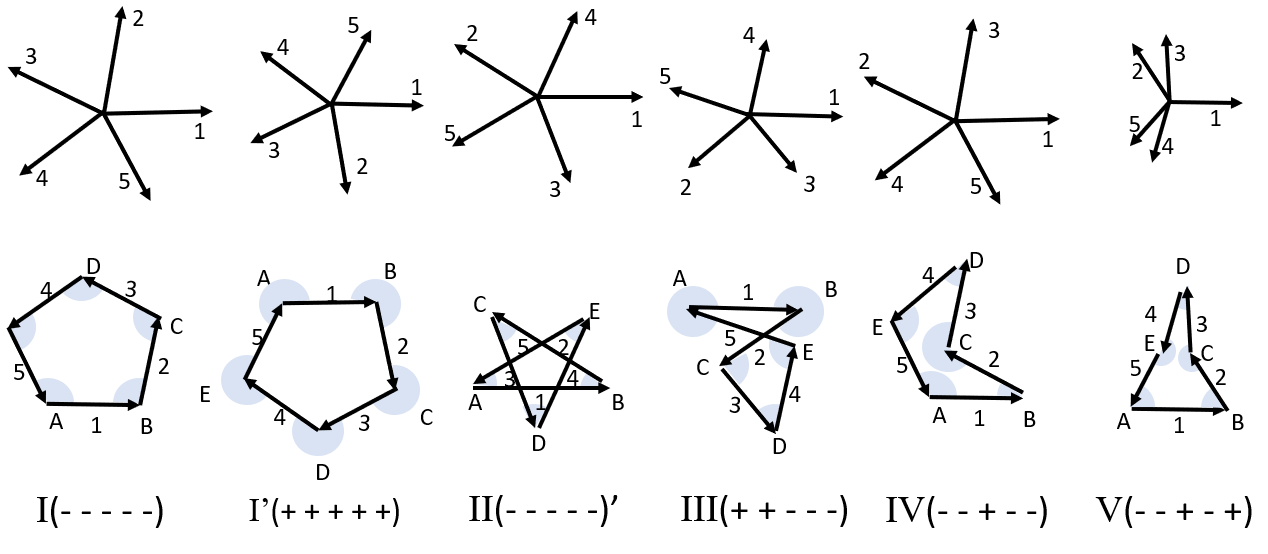}}
\caption{Cells for pentagon}
\label{fpentagon}
\end{figure}

Note that switching all signs corresponds to reversing the cyclic
order, while a cyclic shift in the sequence corresponds to a
cyclic shift in the labelling. Thus (in the order in which the
appear from left to right in Figure \ref{fpentagon}) we have\vsm :

\noindent I.\  One (pentagonal-shaped) cell for the convex pentagon \w{(-----)}
(as in Figure \ref{fmmmmmsym} below);

\noindent I'.\  An analogous (pentagonal-shaped) cell for the convex pentagon
\w[;]{(+++++)}

\noindent II.\  One (pentagonal-shaped) cell for each of the pentagrams \w{(-----)'}
and \w[,]{(+++++)'}

\noindent III.\  Five (hexagonal-shaped) cells: \w[,]{(++---)} \w[,]{(++---)}
  \w[,]{(-++--)} \w[,]{(--++-)} and \w[,]{(---++)} for the middle type, as in
  Figure \ref{fmmppm} below).

Similarly, we have five (hexagonal-shaped) cells for the reverse order
(which looks identical if we disregard the direction in which the angles are measured).

\noindent IV.\  Five (hexagon-shaped) cells of type \w[,]{(+----)} et cetera
(see Figure \ref{fmmpmm})

\noindent V.\  Five triangular cells of type \w[,]{(++-+-)} et cetera
(see Figure \ref{fmppmp} below)\vsm.

%
%
\begin{figure}[ht]
\centering{\includegraphics[width=0.6\textwidth]{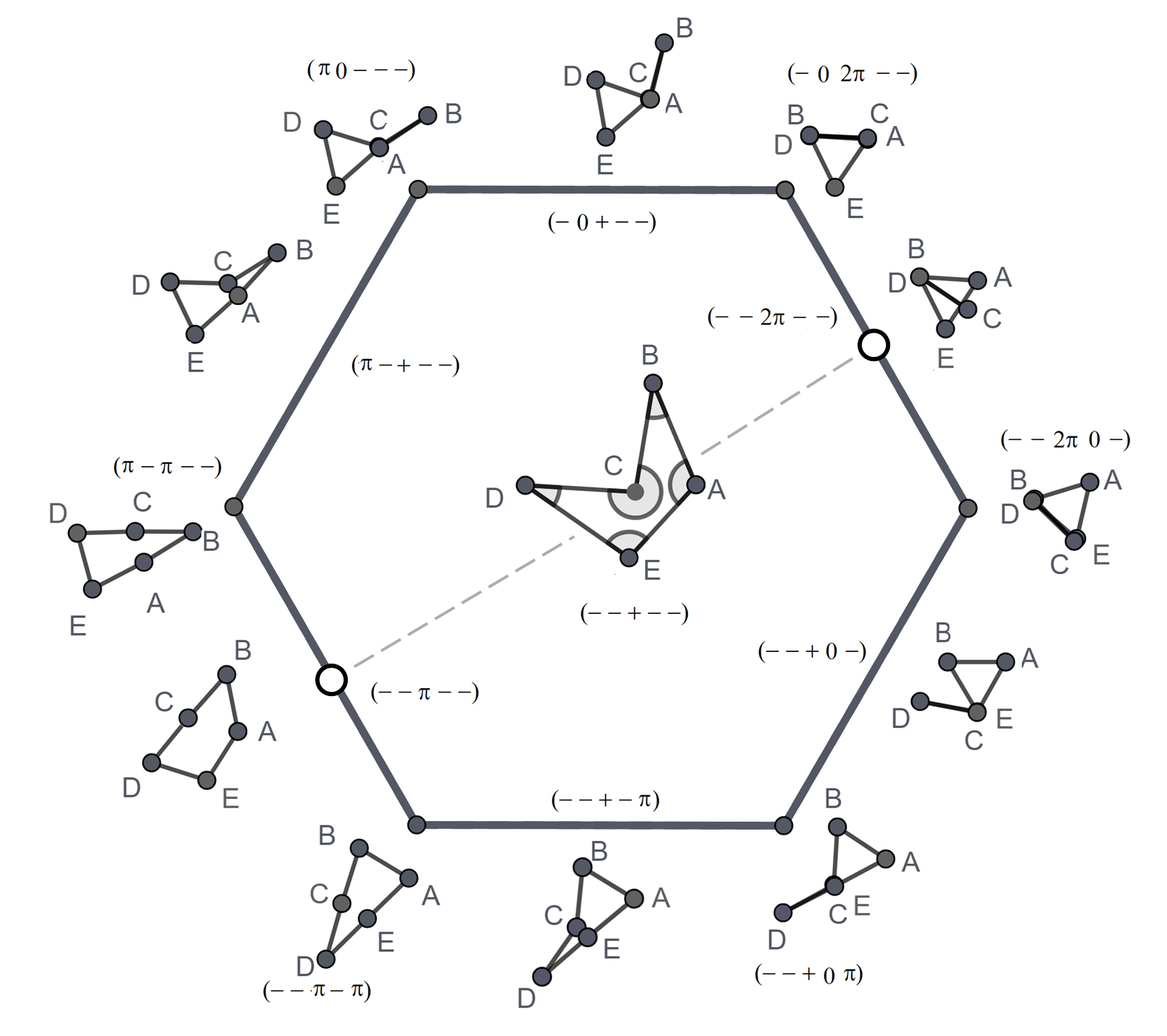}}
\caption{Cell $(--+--)$}
\label{fmmpmm}
\end{figure}
\end{subsection}

\begin{subsection}{Boundaries of the cells}\label{sboundc}
  As noted above, one should think of the 32 cells constituting the reduced \cspace\
  for the equilateral pentagon as polygonal cells (triangles, pentagons, or hexagons),
  identified along common edges. The edges of each such polygonal cell are obtained
  by a collineation:
either straightening one of the angles to $\pi$, or folding it to $0$ (if it was a $-$)
or \w{2\pi} (if it was a $+$). Each vertex $P$ of a cell is obtained by a double
collineation, corresponding to the two edges meeting at $P$.

It is possible to describe explicitly the rules for the allowable collineations and double
collineations (for example, one cannot have two adjacent straightenings), but as these are
particular to the case \w[,]{n=5} we leave them to the reader, illustrating them in
Figures \ref{fmmpmm}, \ref{fmppmp}, \ref{fmmppm}, and \ref{fmmmmmsym}.
\end{subsection}

\begin{subsection}{Symmetries of the equilateral pentagon}\label{ssymep}
  Each of the five types of cells corresponding to the configurations in
  Figure \ref{fpentagon} are exchanged among themselves by the obvious cyclic rotations
  or reflections of the vertices, so only one of each type is needed for the fundamental
  domain of the symmetric \cspace.
However, there are also symmetries acting on each cell. For example,
the dashed line across the hexagon in Figure \ref{fmmpmm} represents an axis of symmetry,
and indeed the two halves of the hexagon are exchanged under the reflection fixing $C$, with
\w{A\leftrightarrow E} and \w[.]{B\leftrightarrow D} The upper left half of the hexagon
corresponds to the linear ordering of the angles
\w[,]{\gamma>\alpha>\varepsilon>\delta>\beta} while the lower right half corresponds
to \w[,]{\gamma>\alpha>\varepsilon>\beta>\delta}

The cell for \w{(-----)} is a pentagon, but in this case there is a
tenfold symmetry \wh as shown in Figure \ref{fmmmmmsym}. Here each
triangular slice of the pentagonal cell corresponds to a certain
linear ordering of the angles of our equilateral pentagon
\w{\Gamma=\Gen{5}} (shown on the right of Figure \ref{fmmmmmsym}).
Each triangular section has one vertex at the center of the cell
(the regular pentagon configuration), one at a vertex of cell
(corresponding to two collineations), and one at the unique
(isosceles) trapezoid configuration with one collineation, which
is the midpoint of an edge. The dashed sides of each slice are
obtained by changing one inequality to an equality. Compare to the subdivided
bipyramid on the right in Figure \ref{fhextetr}.

%
%
\begin{figure}[ht]
\centering{\includegraphics[width=0.7\textwidth]{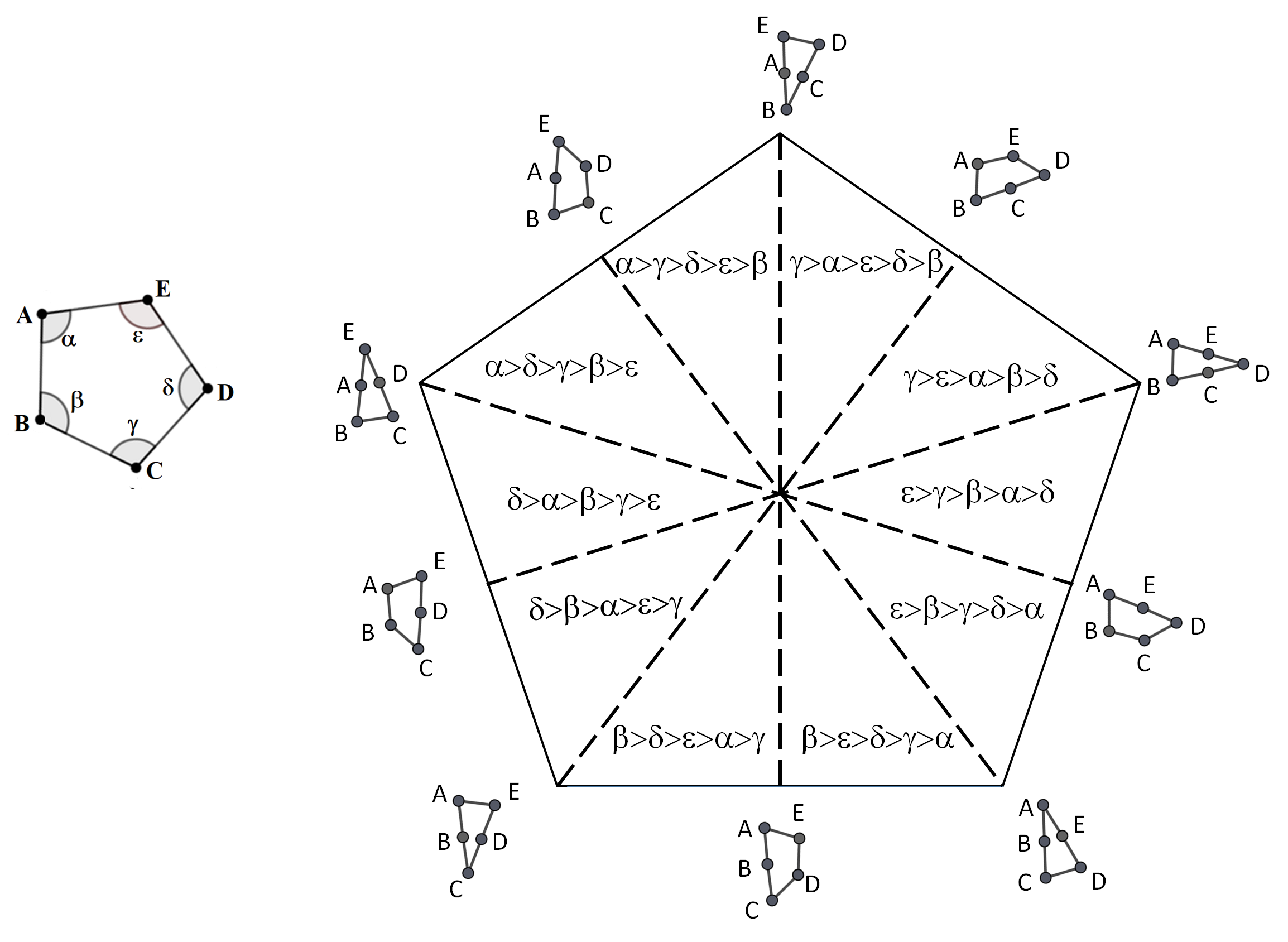}}
\caption{Symmetries of the cell for $(-----)$}
\label{fmmmmmsym}
\end{figure}

The cell for \w{(-----)'} (second from the left in Figure \ref{fpentagon}) is also a
pentagon, similarly divided into 10 triangular regions, as shown in
Figure \ref{fmmmmmosym}.

%
%
\begin{figure}[ht]
\centering{\includegraphics[width=0.6\textwidth]{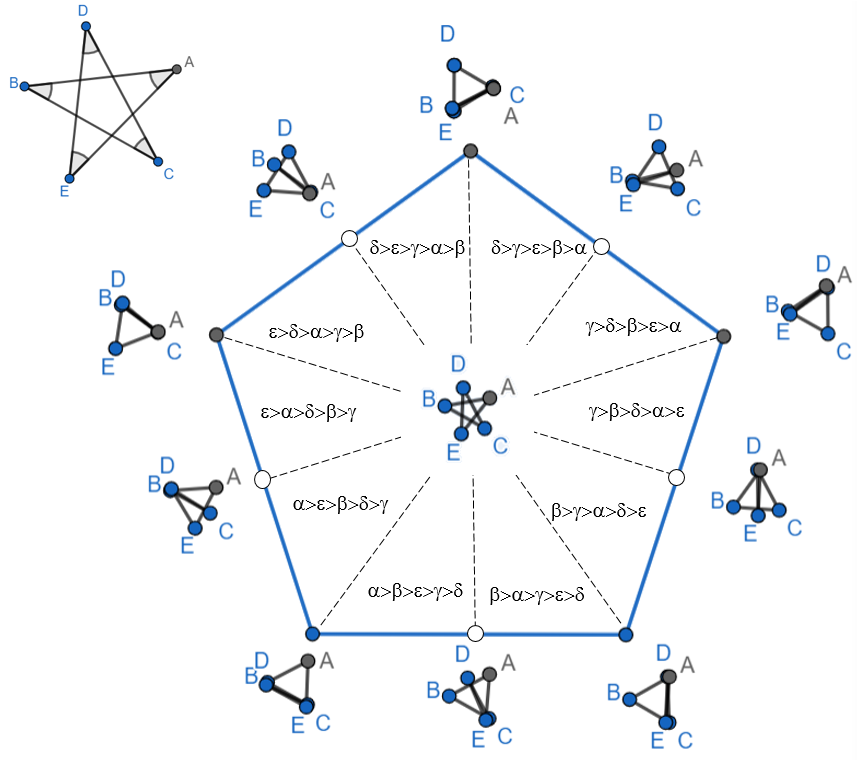}}
\caption{The cell for $(-----)'$}
\label{fmmmmmosym}
\end{figure}

The cell for \w{(++---)}
(third from the left) is pentagonal, divided into two halves (see Figure \ref{fmmppm}),
but with the bisector connecting a vertex to the midpoint of the edge opposite.

%
%
\begin{figure}[ht]
\centering{\includegraphics[width=0.7\textwidth]{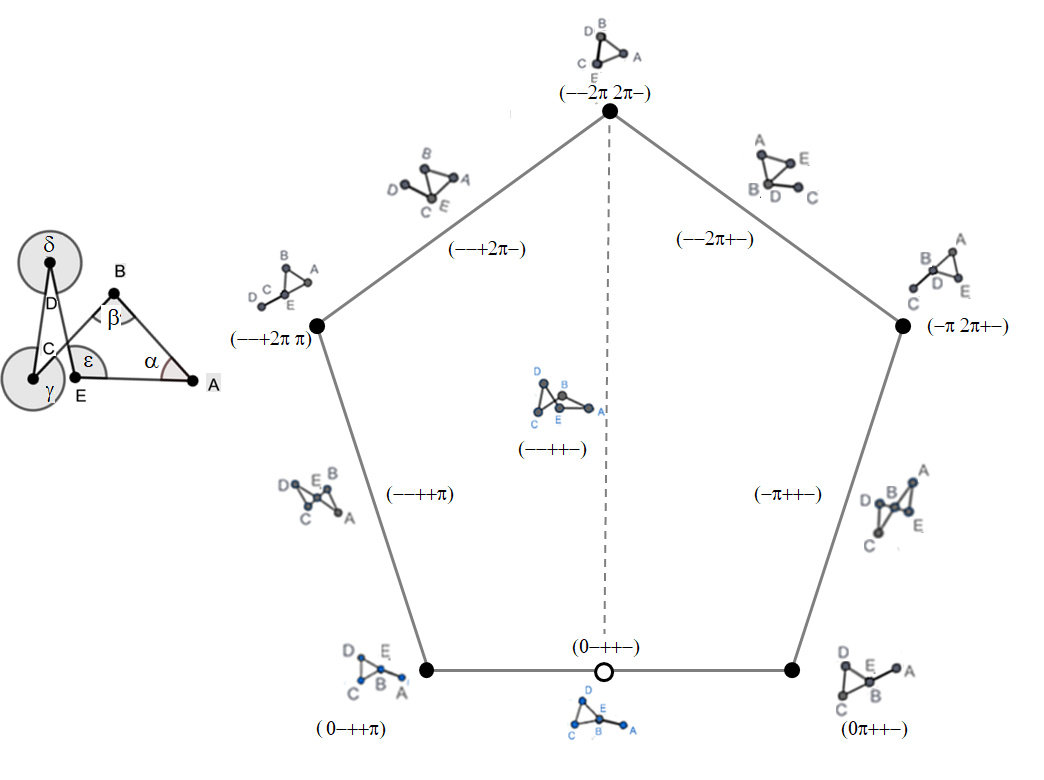}}
\caption{The cell for $(--++-)$}
\label{fmmppm}
\end{figure}

The cell for \w{(++-+-)} (on the right) is a triangle, similarly subdivided into two
halves, as in Figure \ref{fmppmp}.
%
%
\begin{figure}[ht]
\centering{\includegraphics[width=0.4\textwidth]{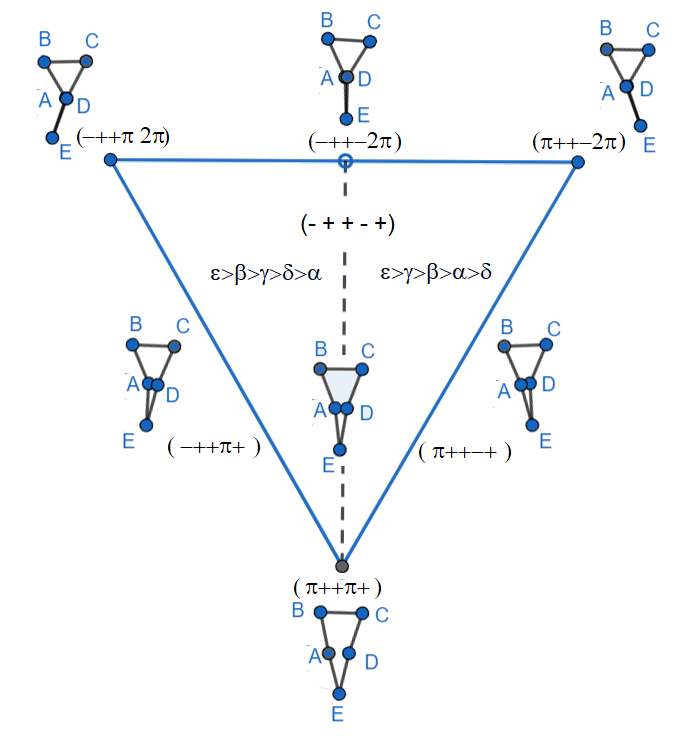}}
\caption{The cell for $(-++-+)$}
\label{fmppmp}
\end{figure}
\end{subsection}

\begin{subsection}{The symmetric configuration space of the pentagon}\label{sscfp}
  From the above discussion we see that a fundamental domain $\Fc$ for the action of
  \w{\Aut(\Gn{5})} on the reduced \cspace\ \w{\hC(\Gn{5})}, depicted in
  Figure \ref{ffund}, is the union of:
%
%
\begin{figure}[ht]
\centering{\includegraphics[width=0.8\textwidth]{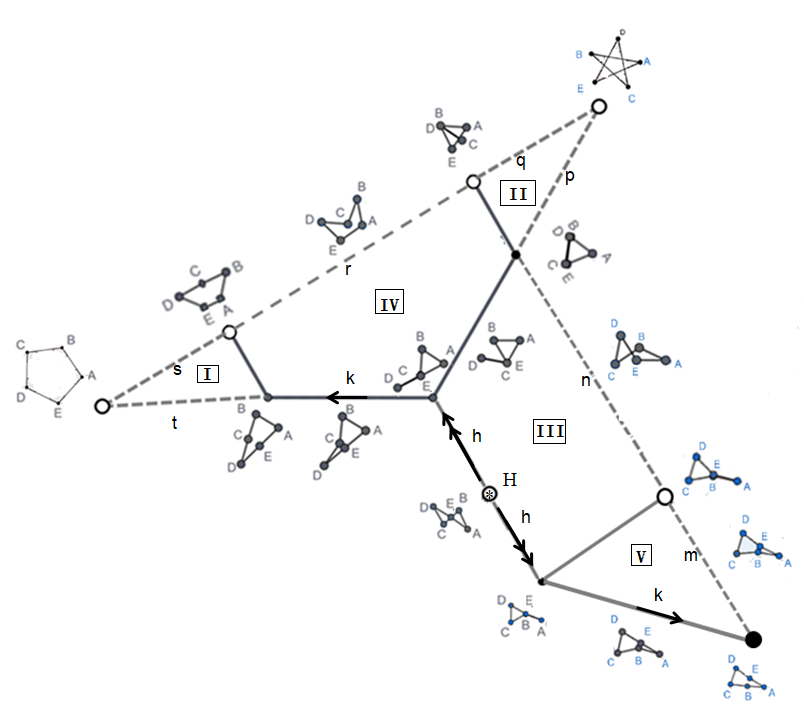}}
\caption{Fundamental domain for pentagon}
\label{ffund}
\end{figure}

\begin{enumerate}
\renewcommand{\labelenumi}{(\roman{enumi})}
\item One half of the hexagonal cell for \w{(--+--)} of Figure \ref{fmmpmm}, marked IV.
\item Attached to it along a half-edge we have one tenth of the pentagonal cell
  for \w{(-----)} of Figure \ref{fmmmmmsym}, marked I.
\item Along the opposite half-edge we have another tenth of the analogous pentagonal cell
  for \w[,]{(-----)'} marked II.
\item One full edge of the cell IV is glued to an edge of the
  half-pentagonal cell III for \w{(-++--)} of Figure \ref{fmmppm}.
\item Finally, the half-cell III for \w{(-++--)} is glued along a half-edge to one
  half V of the triangle for \w{(-++-+)} of Figure \ref{fmppmp}.
\end{enumerate}

The boundary of the fundamental domain $\Fc$ consists of two types of segments

\begin{enumerate}
\renewcommand{\labelenumi}{\alph{enumi}.}
\item The two copies of each of the solid edges $h$ and $k$ are identified pairwise
under appropriate symmetries. The point marked $H$ is fixed under the symmetries.
\item Each of the dashed lines $m$, $n$, $p$, $q$, $r$, $s$ and $t$ is an original axis
  of symmetry in Figures \ref{fmmpmm}, \ref{fmppmp}, \ref{fmmppm}, and \ref{fmmmmmsym},
  respectively, so they are also fixed under
  the corresponding symmetries, with points on the other side (in the full cell) reflected
  back into $\Fc$.
\end{enumerate}

Thus we may summarize the results of this section in:
\end{subsection}

\begin{proposition}\label{pscspent}
  The fully reduced symmetric \cspace\ \w{\wSGn{5}} of the equilateral pentagon
  in the plane is homeomorphic to a closed disc.
\end{proposition}

\begin{proof}
The fundamental domain $X$ in Figure \ref{ffund} is a subspace of the fully reduced
\cspace\ \w[,]{\wCGn{5}} with \w{i:X\to\wCGn{5}} the inclusion.
If \w{p:\wCGn{5}\to\wSGn{5}} is the quotient map, we see that
\w{p\circ i:X\to\wSGn{5}} is surjective, and one-to-one except along the intervals
marked $h$ and $k$ in Figure \ref{ffund}. Thus if \w{r:X\to\widehat{X}} is the quotient
map identifying the two copies of $h$ and $k$ respectively, we see that \w{p\circ i} induces
a homoeomorphism \w{\varphi:\widehat{X}\to\wSGn{5}} (since the closed
disc \w{\widehat{X}} is compact).
\end{proof}

\appendix
%
%
\section{Configuration spaces for quadrilaterals}
\label{cquad}

As noted above, the usual \cspace s of planar quadrilaterals are well known (see, e.g.,
\cite[\S 1]{FarbT}). However, we need their detailed description in order to analyze
the symmetric \cspace s. Thus, in this Appendix we prove Theorem \ref{tscsq}
by considering separately the six cases of \S \ref{socsq}\vsm:

\noindent\textbf{I.}\ The isosceles quadrilateral case:

When \w{\Gamma=ABCD}  is a quadrilateral in the plane with
opposite edges \w{AB} and \w{CD} of equal length, we may
parameterize the configurations $\bx$ of $\Gamma$ as in \S
\ref{spcsp} by a subset of the angles at the four vertices, by
choosing \w{\angle BAD} (from \w{\vec{AD}} to \w[),]{\vec{AB}} and
\w{\angle CDA} (from \w{\vec{DC}} to \w[,]{\vec{DA}} both measured
counter clockwise). We think of \w{\phi:=\angle BAD} as the basic
continuous parameter, and note that to each value of $\phi$ we
associate two values of \w[,]{\angle CDA} corresponding to the
elbow up/down position of $C$ \w[(]{\vare=\pm1} in \S \ref{spcsp})
\wh see Figure \ref{ffourbar}. Note that the precise rule for
calculating \w{\phi'} and \w{\phi''} from $\phi$ is complicated to
state.

%
%
\begin{figure}[ht]
\centering{\includegraphics[width=0.8\textwidth]{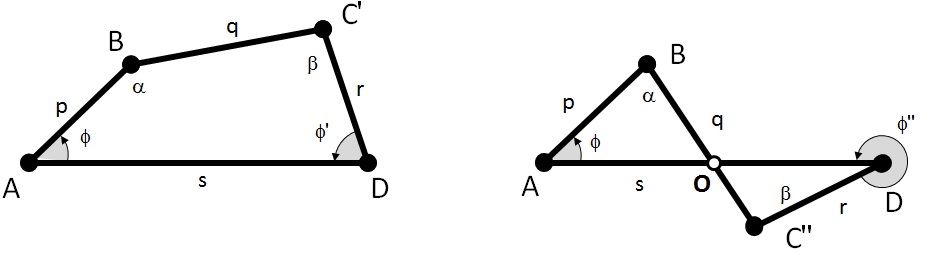}}
\caption{Parameterizing a quadrilateral} \label{ffourbar}
\end{figure}

We need to understand the
action of the \wwd{C\sb{2}}symmetry of \w{\TG} (generated by the graph automorphism
\w{f:\TG\to\TG} given by \w{A\leftrightarrow D} and \w{B\leftrightarrow C} on a
configuration \w[).]{\bx\in\hCG} In the language of \S \ref{ssymnor}, our permutation $\sigma$ is
given by \w[,]{\binom{1234}{4321}} which reverses cyclic orientation, so
\w{\bx=(\phi,\alpha,\beta,\phi')} maps under $f$ to
\w{\by=N(\hx\circ f)=(2\pi-\phi',2\pi-\beta,2\pi-\alpha,2\pi-\phi)}
(where $\alpha$ and $\beta$ are extraneous to determining the configuration,
and may therefore be dropped). Thus the action of \w{C\sb{2}=\AG} on \w{\hCG} takes
\w[,]{(\phi,\phi')} to \w{(2\pi-\phi',2\pi-\phi)} as in Figure \ref{ffourbardual},
with fixed points when \w[.]{\phi+\phi'=2\pi}

%
%
\begin{figure}[ht]
\centering{\includegraphics[width=0.8\textwidth]{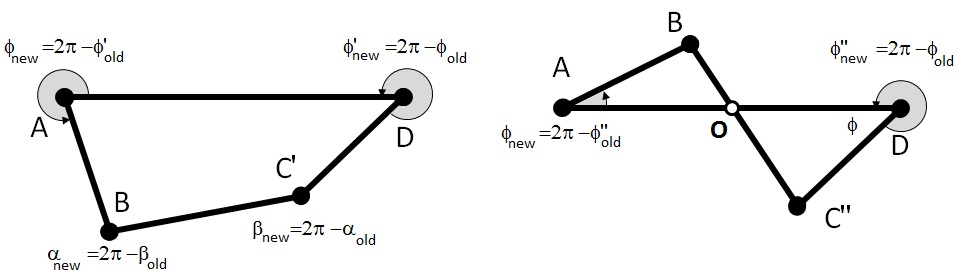}}
\caption{Action of \ww{C\sb{2}} on Figure \protect{\ref{ffourbar}}}
\label{ffourbardual}
\end{figure}

Case \ref{socsq}(i) when $\Gamma$ is isosceles then may be
described as follows:

\begin{lemma}\label{lisoctrap}
  In the notation of \S \ref{nisoctrap}, assume that \w[,]{s>p,q,r} \w[,]{p=r} \w[,]{s<p+q+r} and
  \w[.]{s+q>2p} Then \w{\hCG} is a circle, the \wwd{C\sb{2}}action is the reflection in the diameter
  (with two fixed points), and thus \w{\hSG} is homeomorphic to a closed interval.
\end{lemma}

\begin{proof}
Consider the circle \w{\gamma\sb{B}} of radius \w{p=r}
about a fixed point $A$ in the plane, and another circle \w{\gamma\sb{C}} of the same radius
about a fixed point $D$ at distance \w{\ell\sb{1}} from $A$.  These are the loci of allowable
locations for $B$ and $C$, respectively, if we disregard the requirement that the distance
between them is $q$. By the analysis of \cite{MTrinG} (see also \cite[\S 1]{FarbT}),
the reduced \cspace\ \w{\hCG} can be described as follows in our case:

There is an arc $\zeta$ of the circle \w{\gamma\sb{B}} defined as the intersection
of \w{\gamma\sb{B}} with an annulus $T$ about $D$,
where $T$ consists of the allowable locations for $B$ with respect to
\w[.]{\gamma\sb{C}} Thus the points of $\zeta$ are precisely the possible locations for $B$
satisfying all our constraints.

For each point $B$ of $\zeta$, the circle $\delta$ of radius $q$ about $B$ generically
intersects \w{\gamma\sb{C}} in two points \w{C'}
and \w[,]{C''} corresponding to the elbow up and elbow down positions of
\w{\angle BCD} (except for the two endpoints \w{B\sb{+}} and \w{B\sb{-}} of $\zeta$, for
which \w{C'=C''} and \w{\beta=\angle BCD} is flat).
The angle \w{\angle BAD} is our parameter $\phi$, with \w{\phi'=\angle C'DA}
and \w[.]{\phi''=\angle C''DA}

By the previous discussion, the fixed points of the action of \w{\AG=C\sb{2}} on \w{\hCG} occur
when \w{\phi+\phi'} or \w{\phi+\phi''} equals \w[.]{2\pi} In the case illustrated in
Figure \ref{fisotrapi}, this happens for \w{C'} with its angle \w[,]{\phi'} and the resulting
quadrilateral \w{ABC'D} is not convex (on the right in Figure \ref{ffourbar}).
Here \w{C'} is taken to be the upper intersection point of $\delta$ with
\w[,]{\gamma\sb{C}} and \w{C''} is the lower point.  Note that the two circles \w{\gamma\sb{B}}
and \w{\gamma\sb{C}} may in fact intersect under our hypotheses as in Figure \ref{fisotrapii}
(if \w[),]{q<2p} but this does not affect the argument.

%
%
\begin{figure}[ht]
\centering{\includegraphics[width=0.6\textwidth]{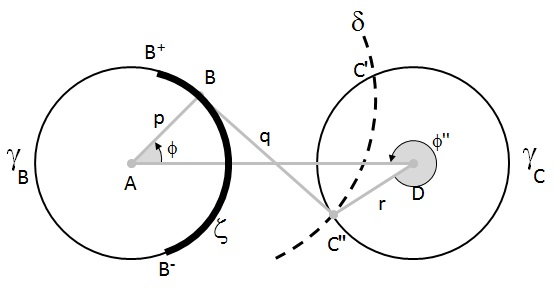}}
\caption{Configuration space for isosceles quadrilateral in case
(i)} \label{fisotrapi}
\end{figure}

The constraints on $p$, $q$, $r$, and $s$ ensure that there is a unique
parallelogram \w{ABDC''}
with opposite sides \w{p=r} and diagonals $q$ and $s$.
The angles \w{\phi=\angle BAD} and \w{2\pi-\phi''=\angle ADC''} are then equal,
and the non-convex quadrilateral \w{ABC''D} is then an allowable configuration
for $\Gamma$.
Since the same argument works replacing \w{C'} by \w[,]{C''} we have exactly two
fixed points for the action of \w{C\sb{2}} on \w[.]{\hCG}

The single configuration associated to the end \w{B\sp{+}} of $\zeta$ is the triangle \w{AB\sp{+}D}
with \w{B\sp{+}CD} aligned (of length \w[).]{q+r} From Figure \ref{ffourbardual} we see that
the \wwd{C\sb{2}}action takes it to the triangle \w{AC\sp{-}D} with \w{AB\sb{\new}C\sp{-}} aligned,
where \w{C\sp{-}} is the lower edge of the arc on \w{\gamma\sb{C}} corresponding to $\zeta$
(not shown in Figure \ref{fisotrapi}), and the point \w{B\sb{\new}} is on the lower half of $\zeta$.
As $B$ moves down from \w{B\sp{+}} along $\zeta$, the point \w{C''} (the lower of the two
intersections of $\delta$ with \w[)]{\gamma\sb{C}} moves down, until it reaches the $x$ axis
\w{AD} (creating the point $O$ on the right in Figure \ref{ffourbar}). We see that at this instance
\w{B\sb{\new}} is on the $x$ axis, and thereafter \w{B\sb{\new}} will be in the upper
half of $\zeta$.

Since \w{\hCG} is a simple
  closed curve in the torus \w[,]{S\sp{1}\times S\sp{1}} and the \wwd{C\sb{2}}action
  is topologically equivalent to the reflection of the circle in a transverse line, we deduce that
  \w{\hSG} is topologically a closed interval.
  \end{proof}

\noindent\textbf{II.}\ Case \ref{socsq}(ii) is essentially a special case of the above:

\begin{lemma}\label{lisoctrapdeg}
  In the notation of \S \ref{nisoctrap}, assume that \w[,]{s>p,q,r} \w[,]{p=r} \w[,]{s<p+q+r} and
  \w[.]{s+q=2p} Then \w{\hCG} is a wedge of two circles, the \wwd{C\sb{2}}action switches the
  two circles between them (fixing the common point), and thus \w{\hSG} is homeomorphic
  to a circle.
\end{lemma}

%
%
\begin{figure}[ht]
\centering{\includegraphics[width=0.4\textwidth]{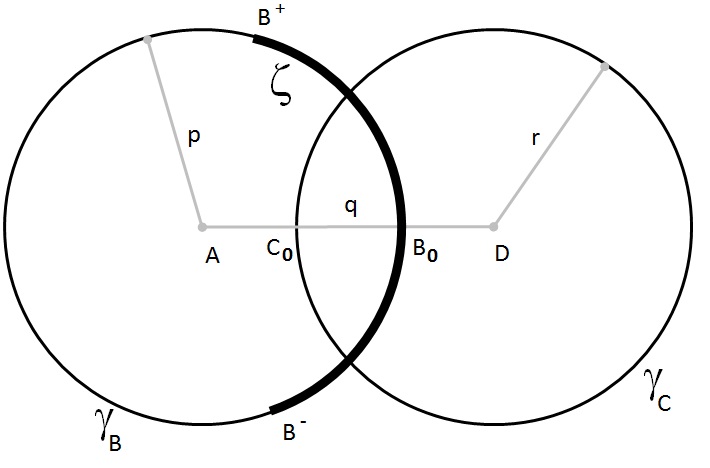}}
\caption{Configuration space for isosceles quadrilateral in case
(ii)} \label{fisotrapii}
\end{figure}

\begin{proof}
  In Figure \ref{fisotrapii} the circle $\delta$ of radius $q$ about the midpoint \w{B\sb{0}} of the
  arc $\zeta$ now intersects \w{\gamma\sb{C}} in a single point \w{C\sb{0}} (which also
  holds for its endpoints \w{B\sp{+}} and \w[,]{B\sp{-}} as before).
  The fully aligned degenerate trapezoid \w{AB\sb{0}C\sb{0}D} is fixed by the
  \wwd{C\sb{2}}action. Moreover, we see that the non-convex parallelograms corresponding to
  the fixed points described in the proof of Lemma \ref{lisoctrap} are both identified with this
  degenerate trapezoid, since we have \w{AO=\frac{s}{2}} and \w{BO=\frac{q}{2}}
  on the right in Figure \ref{ffourbar}, and their sum equals \w{AB=p} by our assumption.

  The argument in the proof of Lemma \ref{lisoctrap} shows that the two circles
  corresponding to the sub-arcs \w{B\sp{+}B\sb{0}} and \w{B\sp{-}B\sb{0}} of $\zeta$ are
  exchanged under the \wwd{C\sb{2}}action.
\end{proof}

\noindent\textbf{III.} \ Case \ref{socsq}(iii) becomes:

\begin{lemma}\label{lisoctrapdg}
In the notation of \S \ref{nisoctrap}, assume that \w[,]{s>p,q,r} \w[,]{p=r} \w[,]{s<p+q+r} and
\w[.]{s+q<2p} Then \w{\hCG} is a disjoint union of two circles, the \wwd{C\sb{2}}action switches the
  two circles between them, and thus \w{\hSG} is homeomorphic
  to a circle.
\end{lemma}

%
%
\begin{figure}[ht]
\centering{\includegraphics[width=0.4\textwidth]{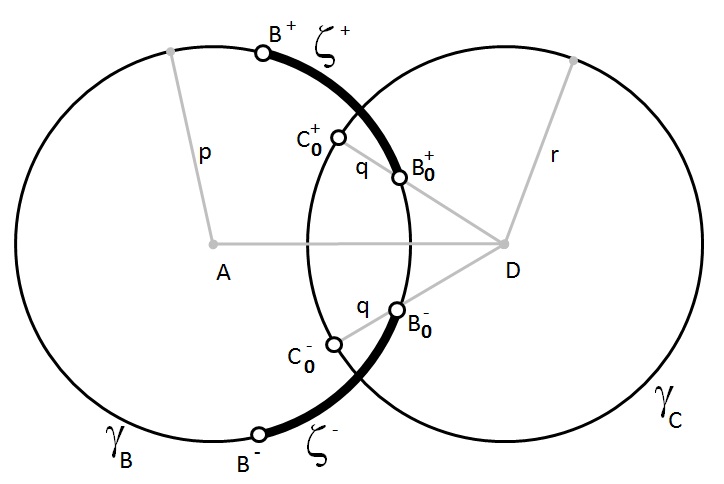}}
\caption{Configuration space for isosceles quadrilateral in case
(iii)} \label{fisotrapiii}
\end{figure}

\begin{proof}
  In this case the $\zeta$ of Figures \ref{fisotrapi}-\ref{fisotrapii} splits into two disjoint arcs
  \w{\zeta\sp{+}=B\sp{+}B\sp{+}\sb{0}} and \w[,]{\zeta\sp{-}=B\sp{-}B\sp{-}\sb{0}}
as in Figure \ref{fisotrapiii}, and the proof of Lemma \ref{lisoctrapdeg} shows that
  the non-convex parallelogram corresponding to a possible fixed point of the
  \wwd{C\sb{2}}action cannot exist. This action simply switches the two disjoint circles of \w{\hCG}
  between them, as before.
\end{proof}

\noindent\textbf{IV.} \ Case \ref{socsq}(iv) is similar:

\begin{lemma}\label{lisoctrapd}
If \w{s=q\geq p>r} in the notation of \S \ref{nisoctrap}, \w{\hCG} is a disjoint union of two circles,
  the \wwd{C\sb{2}}action switches them, and \w{\hSG} is again a circle.
\end{lemma}

%
%
\begin{figure}[ht]
\centering{\includegraphics[width=0.3\textwidth]{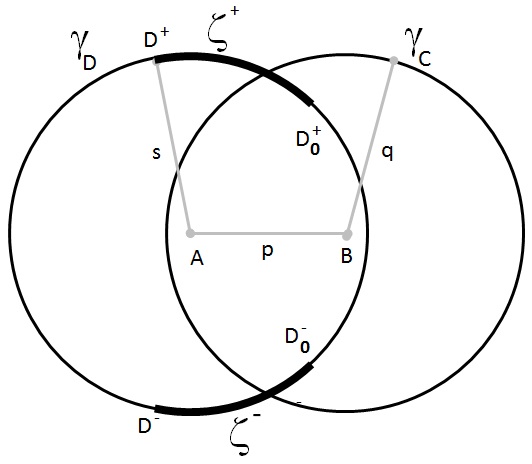}}
\caption{Configuration space for isosceles quadrilateral in case
(iv)} \label{fisotrapiv}
\end{figure}

\begin{proof}
  In Figure \ref{fisotrapiv} we choose to draw the two congruent circles \w{\gamma\sb{C}}
  and \w{\gamma\sb{D}} about $A$ and $B$, so that $\zeta$ again splits into two disjoint arcs
  \w{\zeta\sp{+}=D\sp{+}D\sp{+}\sb{0}} and \w[,]{\zeta\sp{-}=D\sp{-}D\sp{-}\sb{0}}
  as in Figure \ref{fisotrapiii}, and once more a non-convex parallelogram corresponding
  to a possible fixed point of the \wwd{C\sb{2}}action cannot exist. This action again switches the
  two disjoint circles of \w{\hCG} between them.
\end{proof}

\noindent\textbf{V.} \ The case of a parallelogram:

  When \w{\Gamma=ABCD} is a parallelogram, the description of \textbf{IV} should
  be modified as follows: specializing the description of Figure \ref{ffourbar}
  as in Figure \ref{fpfourbar}

%
%
\begin{figure}[ht]
\centering{\includegraphics[width=0.8\textwidth]{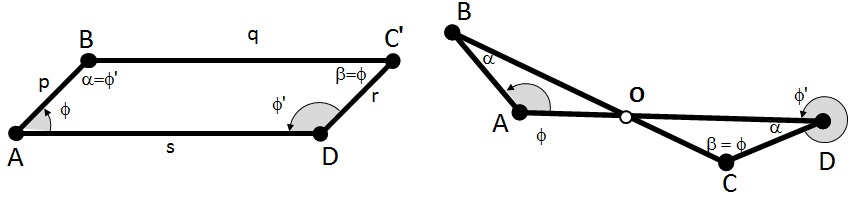}}
\caption{Parameterizing a parallelogram} \label{fpfourbar}
\end{figure}
we now a \wwd{C\sb{2}\sp{(1)}\times C\sb{2}\sp{(2)}}symmetry generated by two graph
automorphisms: the first \wwd{C\sb{2}}action is given by \w{A\leftrightarrow D} and \w[,]{B\leftrightarrow C} and the second by \w{A\leftrightarrow B} and
\w[.]{C\leftrightarrow D}
It turns out that the first two coincide on the left (convex) configuration of Figure \ref{fpfourbar}, yielding the left hand side of Figure \ref{fpfourbardual}.
On the other hand, the first \wwd{C\sb{2}}action on the right (non-convex) configuration in
Figure \ref{fpfourbar} yields the upper right hand in Figure \ref{fpfourbardual}, while the
second yields the lower right hand quadrilateral there.

%
%
\begin{figure}[ht]
\centering{\includegraphics[width=0.8\textwidth]{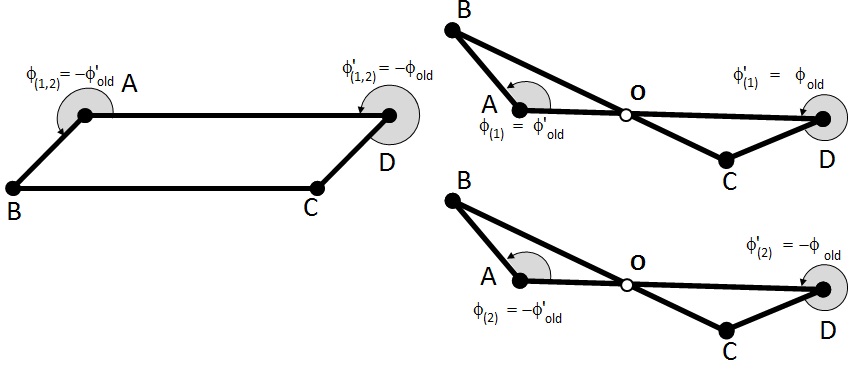}}
\caption{Action of \ww{C\sb{2}\times C\sb{2}} on Figure \protect{\ref{fpfourbar}}}
\label{fpfourbardual}
\end{figure}

\begin{lemma}\label{lisoctra}
  If in the notation of \S \ref{nisoctrap} \w{s=q> p=r} (a parallelogram),
  then \w{\hCG} is a union
  of four arcs $x$, $y$, $z$, and $w$ with ends glued at $G$ and $H$, respectively,
  as depicted in Figure \ref{fparallelogram}. The first \wwd{C\sb{2}}action sends $z$
  antipodally to $w$, while the second
  \wwd{C\sb{2}}action reflects the left half \w{x\sb{1}} of $x$ to the right half \w{x\sb{2}}
  (with a fixed point at their common end $J$), and reflects the left half \w{y\sb{1}} of $y$
  to the right half \w{y\sb{2}}
  (with a fixed point at their common end $K$), thus identifying $G$ with $H$.
  As a result, \w{\hSG} is a wedge of the circle \w{w\sim z} and a segment.
\end{lemma}

%
%
\begin{figure}[ht]
\centering{\includegraphics[width=0.3\textwidth]{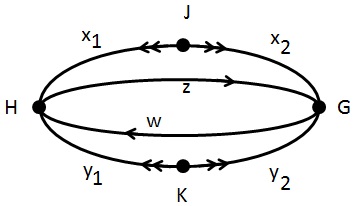}}
\caption{Configuration space for the parallelogram of case (v)}
\label{fparallelogram}
\end{figure}

\begin{proof}
  In Figure \ref{fisotrapi} we see that $\zeta$ equals the whole circle \w[,]{\gamma\sb{B}}
  and at the two points of intersection of $\zeta$ with the $x$-axis (corresponding to $G$ and
  $H$ in Figure \ref{fparallelogram}) we have \w[,]{C'=C''} which yields the description of
  \w[.]{\hCG} The actions of the two cyclic groups \w{C\sb{2}\sp{(1)}} and \w{C\sb{2}\sp{(2)}}
  follow from the description in Case \textbf{V}.
\end{proof}

\begin{lemma}\label{lrhombus}
  If in the notation of \S \ref{nisoctrap} \w{s=p>q=r} (a deltoid) then \w{\hCG} is a union
  of four arcs $x$, $y$, $z$, and $w$ with ends glued at $G$ and $H$, respectively,
  as depicted in Figure \ref{fdeltoid}, with the \wwd{\AG=C\sb{2}}action fixing the arcs
  $x$ and $y$ pointwise and reflecting the arc $z$ and $w$ to each other in the diameter
  \w[.]{GH} Thus \w{\hSG} consists of three arcs with left endpoints glued at $H$ and the
  right endpoints glued at $G$.
\end{lemma}

%
%
\begin{figure}[ht]
\centering{\includegraphics[width=0.3\textwidth]{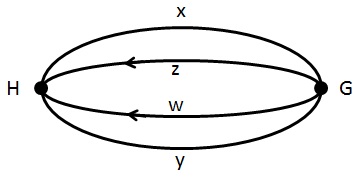}}
\caption{Configuration space for the parallelogram of case (v)}
\label{fdeltoid}
\end{figure}

\begin{proof}
The analysis of Lemma \ref{lisoctrap} shows that we have an arc $\zeta$ of the circle
\w{\gamma\sb{B}} with a single corresponding point \w{C'=C''} on \w{\gamma\sb{C}} for its two
endpoints \w{B\sp{+}} and \w[,]{B\sp{-}} and otherwise two distinct values.  These yield the
two arcs $w$ and $z$ with the \wwd{C\sb{2}}action as described.
However, the fact that \w{s=p} implies that $\zeta$ passes through $D$, at which point the
edges \w{AB} and \w{AD} coincide, so \w{BC} and \w{CD} coincide, too, and this common edge
is free to rotate about $D$, yielding the arcs $x$ and $y$.
\end{proof}

\noindent\textbf{VI.} \ The square:

Recall that if \w{\Gamma=\Gen{4}} is an equilateral quadrilateral, the automorphism group
\w{\AG} is the dihedral group \w{D\sb{4}} generated by the rotation $R$ (given by
\w[,]{A\mapsto B} \w[,]{B\mapsto C} \w[,]{C\mapsto D} and \w[),]{D\mapsto A} of order $4$,
and the reflection $T$ (given by \w{A\leftrightarrow C} with $B$ and $D$ fixed).

\begin{lemma}\label{lsquare}
  In the equilateral case \wb[,]{s=p=q=r} \w{\hCG} is a union of three circles
  of four arcs $x$, $y$, $z$, and $w$ with ends glued at $G$, $H$, and $L$,
  as depicted in Figure \ref{fsquare}. The reflection $T$ sends $x$ to $y$,
 \w{x'} to \w[,]{y'} $u$ to $v$, and fixes $z$ and $w$ pointwise.
 The rotation $R$ sends $x$ to \w[,]{x'} $y$ to \w[,]{y'} $u$ to $z$, $z$ to $v$, $v$ to $w$
 and $w$ to $y$, and fixes $J$, $K$ and $L$.  Thus \w{\hSG} consists of
  two arcs corresponding to $x$ and $z$, glued at their common endpoint $H$.
\end{lemma}

%
%
\begin{figure}[ht]
\centering{\includegraphics[width=0.3\textwidth]{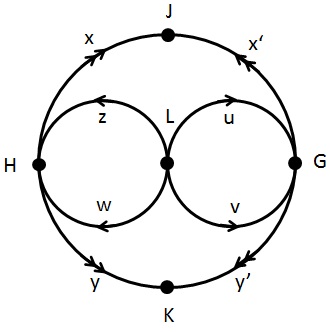}}
\caption{Configuration space for the rhombus}
\label{fsquare}
\end{figure}

\begin{proof}
The analysis of Lemma \ref{lisoctrap} shows that we have two
circles \w{\gamma\sb{B}} and \w{\gamma\sb{C}} with the same radius
$\ell$ about $A$ and $D$ respectively, with \w[.]{|AD|=\ell} To a
point $B$ on \w{\gamma\sb{B}} (with \w{\angle DAB=\phi} as
parameter) there correspond two points on \w[:]{\gamma\sb{C}} one
being $A$ itself (so forming a degenerate quadrilateral) and the
other, $C$, forming a parallelogram, so \w{\angle ADC=\phi'}
satisfies \w[).]{\phi+\phi'=\pi} The parallelogram case
corresponds to the circle \w{xx'yy'} in Figure \ref{fsquare}, with
$H$ at \w{(\phi,\phi')=(\pi,0)} and $G$ at
\w[).]{(\phi,\phi')=(0,\pi)} The degenerate case with \w{\phi'=0}
corresponds to the circle \w[,]{zw} with $L$ at
\w[,]{(\phi,\phi')=(0,0)} while the case \w{\phi=0} corresponds to
the circle \w[.]{uv}

The reflection $T$ takes \w{(\phi,\phi')} to \w[,]{(-\phi',-\phi)}
while the rotation $R$ takes \w{(\phi,\phi')} to
\w[,]{(\phi',\phi)} unless \w{\phi'=0} in which case
\w[.]{(\phi,0)\mapsto(0,-\phi)} Note that the two rules are
consistent at $H$ and $L$.
\end{proof}

%
%


\begin{thebibliography}{ABCD}
%
\bibitem[BGRT]{BGRTamaH}
I.~Basabe, J.~Gonz{\'{a}}lez, Y.B.~Rudyak, D.~Tamaki,
``Higher topological complexity and its symmetrization'',\hsm
\textit{Alg.\ Geom.\ Topology} \textbf{14} (2014), pp.~2103-2124.
%
\bibitem[BR]{BReciT}
A.~Bianchi \& D.~Recio-Mitter,
``Topological complexity of unordered configuration spaces of surfaces'',\hsm
\textit{Alg.\ Geom.\ Topology} \textbf{19} (2019), pp.~1359-1384.
%
\bibitem[BS1]{BShvaG}
D.~Blanc \& N.~Shvalb,
``Generic singular configurations of linkages'',\hsm
\textit{Top.\ \& Applic.} \textbf{159} (2012), pp.~877-890.
%
\bibitem[BS2]{BShvaC}
D.~Blanc \& N.~Shvalb,
``Configuration spaces of spatial linkages: Taking Collisions Into Account'',\hsm
\textit{Bull.~Kor.~Math.~Soc.} \textbf{54} (2017), pp.~2183-2210.
%
\bibitem[BK]{BKaluET}
Z.~B{\l}aszczyk \& M.~Kaluba,
``Effective topological complexity of spaces with symmetries'',\hsm
\textit{Pub.\ Mat.} \textbf{62} (2018), pp.~55-74.
%
\bibitem[C]{DCoheTC}
D.C.~Cohen,
``Topological complexity of classical configuration spaces and related objects'',\hsm
in \textit{Topological complexity and related topics}, Contemp.\ Math.\ \textbf{702},
AMS, Providence, RI, 2018, pp.~41-60.
%
\bibitem[CDR]{CDRoteS}
R.~Connelly, E.D.~Demaine, \& G.~Rote,
``Straightening polygonal arcs and convexifying polygonal cycles'',
\textit{Discrete Comput.\ Geom.} \textbf{30} (2003), pp.~205-239.
%
\bibitem[D]{DavisTC}
D.M.~Davis,
``Topological complexity of some planar polygon spaces'',\hsm
\textit{Bol. Soc. Mat. Mex. (3)} \textbf{23} (2017), pp.~129-139.
%
\bibitem[FH]{FHusGT}
E.R.~Fadell \& S.Y.~Husseini",
\textit{Geometry and topology of configuration spaces},\hsm
%
\bibitem[FN]{FNeuC}
E.R.~Fadell \& L.P.~Neuwirth",
``Configuration spaces'',\hsm
\textit{Math.\ Scand.} \textbf{10} (1962), pp.~111-118
%
\bibitem[F1]{FarbTC}
M.S.~Farber,
``Topological complexity of motion planning'',\hsm
\textit{Discrete Comput.\ Geom.} \textbf{29} (2003), pp.~211-221.
%
\bibitem[F2]{FarbT}
M.S.~Farber,
\textit{Invitation to Topological Robotics},\hsm
European Mathematical Society, Zurich, 2008.
%
\bibitem[FG]{FGranS}
M.S.~Farber \& M.~Grant,
``Symmetric Motion Planning'',\hsm
in \textit{Topology and robotics (Zurich, 2006)}, Contemp.\ Math.\ \textbf{438},
AMS, Providence, RI, 2007, pp.~85-104.
%
\bibitem[FP]{FPaveS}
A.~Franc \& P.~Pave{\v{s}i\'{c}},
``Spaces with high topological complexity'',\hsm
\textit{Proc.\ Roy.\ Soc.\ Edin., Sec.\ A} \textbf{144} (2014), pp.~761-773.
%
\bibitem[GP]{GPaninM}
P.~Galashin \& G.Yu.\ Panina,
``Manifolds associated to simple games'',\hsm
\textit{J.\ Knot Theory Ramif.} \textbf{25} (2016), No.\ 12.
%
\bibitem[G]{GhriC}
R. Ghrist,
``Configuration spaces, braids, and robotics'',\hsm
in \textit{Braids}, World Sci. Publ., Hackensack, NJ, 2010, pp.~263-304.
%
\bibitem[Hal]{HallK}
A.S.~Hall, Jr.,
\textit{Kinematics and linkage design},\hsm
Prentice-Hall, Englewood Cliffs, NJ, 1961.
%
\bibitem[HR]{HRodrS}
J.-C.~Hausmann \& E.~Rodriguez,
``The space of clouds in Euclidean space'',\hsm
\textit{Experiment.\ Math.} \textbf{13} (2004), pp.~31-47.
%
\bibitem[Hav]{HavD}
T.F.~Havel,
``Some examples of the use of distances as coordinates for Euclidean geometry'',\hsm
\textit{J.\ Symbolic Comput.} \textbf{11} (1991), pp.~579-593.
%
\bibitem[Hi]{HiroT}
H.~Hironaka,
``Triangulations of algebraic sets'',\hsm
in \textit{Algebraic geometry (Humboldt State Univ., Arcata, Calif., 1974)},
Proc.\ Sympos.\ Pure Math.\ \textbf{29}, AMS, Providence, R.I., 1975, pp.~165-185.
%
\bibitem[HS]{HSharI}
C.~Housecroft \& A.G.~Sharpe,
\text{Inorganic Chemistry},
5th edition, Pearson International, London, UK, 2018.
%
\bibitem[K]{KamiTE}
Y.~Kamiyama,
``Topology of equilateral polygon linkages'',\hsm
\textit{Top.~Applic.} \textbf{68} (1996), pp.~13-31.
%
\bibitem[KTe]{KTezuT}
Y.~Kamiyama \& M.~Tezuka,
``Topology and geometry of equilateral polygon linkages in the Euclidean plane'',\hsm
\textit{Quart.\ J.\ Math.\ Oxford (2)} \textbf{50} (1999), pp.~463-470.
%
\bibitem[KTs]{KTsukC}
Y.~Kamiyama \& S.~Tsukuda",
``The configuration space of the $n$-arms machine in the Euclidean space'',\hsm
\textit{Top.\ \& Applic.} \textbf{154} (2007), pp.~1447-1464.
%
\bibitem[KM]{KMillM}
M.~Kapovich \& J.~Millson,
``On Moduli Space of Polygons in the Euclidean Plane'',\hsm
\textit{J.\ Diff.\ Geom.} \textbf{42} (1995), pp.~430-464.
%
\bibitem[LLL]{LLLuO}
Fengling Li, Hao Li, \& Zhi L{\"{u}},
``A theory of orbit braids'',\hsm
preprint, 2019 {\tt arXiv:1903.11501}.
%
\bibitem[L]{LojaT}
S. {\L}ojasiewicz,
``Triangulation of semi-analytic sets'',\hsm
\textit{Ann.\ Scuola Norm.\ Sup.\ Pisa Cl.\ Sci.\ (3)} \textbf{18} (1964),
pp.~449-474.
%
\bibitem[Ma]{MayEHC}
J.P.~May,
\textit{Equivariant homotopy and cohomology theory},\hsm
American Mathematical Society, Providence, RI, 1996.
%
\bibitem[Me]{MerlP}
J.~P.~Merlet,
\textit{Parallel Robots},\hsm
Kluwer Academic Publishers, Dordrecht, 2000.
%
\bibitem[MT]{MTrinG}
R.J.~Milgram \& J.~C.~Trinkle,
''The Geometry of Configuration Spaces for Closed Chains in Two and
Three Dimensions'',\hsm
\textit{Homology, Homotopy \& Applic.} \textbf{6} (2004), pp.~237-267.
%
\bibitem[MW]{MWuTC}
A.~Murillo-Mas \& J.~Wu,
``Topological complexity of the work map'',\hsm
\textit{J.\ Top.\ Applic.} \textbf{12} (2021), pp.~219-238.
%
\bibitem[P]{PaninM}
G.Yu.~Panina,
``Moduli space of planar polygonal linkage: a combinatorial description'',\hsm
\textit{Arnold Math.\ J.} \textbf{3} (2017), pp.~351-364.
%
\bibitem[S]{SeliG}
J.M.~Selig,
\textit{Geometric Fundamentals of Robotics},\hsm
Springer-\-Verlag \textit{Mono.\ Comp.\ Sci.}, Berlin-\-New York, 2005.
%
\bibitem[SSBB]{BBSShvaT}
N.~Shvalb, M.~Shoham, H.~Bamberger, \& D.~Blanc,
``Topological and Kinematic Singularities for a Class of Parallel Mechanisms'',\hsm
\textit{Math.\ Prob.\ in Eng.} \textbf{2009} (2009), Art.\ 249349,
pp.~1-12.
%
\bibitem[SSB]{SSBlaCA}
N.~Shvalb, M.~Shoham, \& D.~Blanc,
``The Configuration Space of Arachnoid Mechanisms'',\hsm
\textit{Fund.\ Math.} \textbf{17} (2005), pp.~1033-1042.
%
\bibitem[T]{TsaiR}
L.~W.~Tsai,
\textit{Robot Analysis - The mechanics of serial and parallel manipulators},\hsm
Wiley interscience Publication - John Wiley \& Sons, New York, 1999.
%
\end{thebibliography}
\end{document}